\documentclass[12pt]{amsart}
 
\usepackage[margin=1.25in]{geometry} 
\usepackage{enumerate}
\usepackage[mathscr]{eucal}
\usepackage{color} 
\usepackage[colorlinks=true, linkcolor=red, citecolor=blue, menucolor=black]{hyperref}

\usepackage{amsmath,amsthm, amsfonts, amssymb}
\numberwithin{equation}{subsection}
\usepackage{graphicx}

\usepackage{listings}
\usepackage{enumerate}
\usepackage{centernot}
\usepackage{tikz, tikz-cd}

\usepackage{hyperref}
\theoremstyle{plain}
\newtheorem{introthm}{Theorem}

\newtheorem{theorem}{Theorem}[section]
\newtheorem{lemma}[theorem]{Lemma}
\newtheorem{corollary}[theorem]{Corollary}
\newtheorem{proposition}[theorem]{Proposition}

\theoremstyle{definition}
\newtheorem{definition}[theorem]{Definition}
\newtheorem{remark}[theorem]{Remark}
\newtheorem{example}[theorem]{Example}

\title[McDuff \& Prime von Neumann algebras from Thompson-Like Groups]{McDuff and Prime von Neumann algebras arising from Thompson-Like Groups}

\author[P. DeBonis]{Patrick DeBonis}
\address{Department of Mathematics, Purdue University, Mathematical Sciences Bldg, 150 N University St, West Lafayette, IN 47907, USA.}
\email{pdebonis@purdue.edu}

\author[R. de Santiago]{Rolando de Santiago}
\address{Department of Mathematics \& Statistics, Cal State University Long Beach, 1250 Bellflower Blvd, Long Beach, CA 90840, USA.}
\email{rolando.desantiago@csulb.edu}

\author[K. Khan]{Krishnendu Khan}
\address{Department of Mathematics and Statistics, University of Maine, 323 Neville Hall, Orono, 04469, ME, USA.}
\email{krishnendu.khan@maine.edu}

\begin{document}

\begin{abstract}
     In this paper we show that the cloning system construction of Skipper and Zaremsky \cite{skipper2021almost}, under sufficient conditions, gives rise to Thompson-Like groups which are stable; in particular, these are McDuff groups in the sense of Deprez and Vaes \cite{deprez2018inner}. This answers a question of Bashwinger and Zaremsky posed in \cite{bashwinger2021neumann} in the affirmative. In the opposite direction, we show that the group von Neumann algebra for the Higman-Thompson groups $T_d$ and $V_d$ are both prime II$_1$ factors. This follows from a new deformation/rigidity argument for a certain class of groups which admit a proper cocycle into a quasi-regular representation that is not necessarily weakly $\ell^2$. 
\end{abstract}

\maketitle

\tableofcontents

\section{Introduction}
II$_1$ factors are a central object in the study in the classification of  von Neumann algebras. Given a group $G$ (or a group action on a standard probability space $G\curvearrowright(X, \mu)$), Murray and von Neumann described a natural process to associate a von Neumann algebra  $L(G)$, termed the group von Neumann algebras (or $L^\infty(X)\rtimes G$ termed the group measure space construction),  \cite{murray1937rings,murray1943rings}. In the case when $G$ is an infinite group,  $L(G)$ will be a II$_1$ factor when $G$ is an ICC group, while $L^\infty(X) \rtimes G$ will be a II$_1$ factor when the action is free, ergodic, and probability measure-preserving.  An important problem in this field is to determine what properties of the group, or the group action,  give rise to von Neumann algebraic invariants.

In this regime, the McDuff property for a II$_1$ factor $M$, which states that $M\cong M\bar\otimes \mathcal{R}$ where $\mathcal{R}$ is the unique hyperfinite II$_1$ factor, and its generalizations have provided the tools that enabled us to classify large classes of von Neumann algebras \cite{mcduff1970central}. At the opposite end of the spectrum lies  primeness of II$_1$ factors, which means that $M$ does not admit a decomposition as a tensor product of II$_1$ factors.  In this paper, we investigate the McDuff property for group measure space von Neumann algebras associated to generalizations of the Thompson groups and primeness for the group von Neumann algebra of two of the Higman-Thompson Groups.

\subsection{Statement of the results}
Stefan Witzel and Matthew Zaremsky in \cite{witzel2018thompson} and Rachel Skipper and Zaremsky in \cite{skipper2021almost} developed a framework that both unified several existing generalizations of the Thompson Groups and provides an axiomized approached for constructing new examples. See for example \cite{aroca2022new}, \cite{brin2007algebra}, \cite{dehornoy2006group}, \cite{brady2008pure}, \cite{nekrashevych2004cuntz}, \cite{rover1999constructing}, \cite{bashwinger2021neumann}, \cite{bashwinger2023neumann}. This machinery, referred to as \emph{cloning systems}, is defined and explained in detail in Section \ref{section:CS} (see also \cite{UserGuideCloningZar}).

In \cite{bashwinger2021neumann} Eli Bashwinger and Zaremsky were the first to study the group von Neumann Algebras of these Thompson-Like groups arising from cloning systems. They were able to determine sufficient conditions for a cloning system to satisfy, to ensure that $\mathscr{T}_d(G_*)$ (the notation for a Thompson-Like group defined in Section \ref{section:CS}) is ICC, meaning $L(\mathscr{T}_d(G_*))$ will be a II$_1$ factor.  In particular, they gave a class of groups where $L(\mathscr{T}_d(G_*))$ is a \emph{McDuff} II$_1$ factor.  In the same paper, Bashwinger and Zaremsky asked when is $\mathscr{T}_d(G_*)$ a \emph{McDuff group}? Defined by Deprez and Vaes in \cite{deprez2018inner}, a group $G$ is a \textit{McDuff group}  if it admits an action on a probability space $(X, \mu)$ such that $L^{\infty}(X, \mu) \rtimes G$ is a McDuff II$_1$ factor. We show that the main class of groups studied by Bashwinger and Zaremsky always admit such an action, and thus are McDuff groups. Additionally, there are non-ICC examples of Thompson-Like groups from cloning systems that are McDuff groups.

\begin{introthm}[Theorem \ref{thmstb}]\label{MainThrm:McDuff}
    Let $((G_n)_{n \in \mathbb{N}}, (\rho_n)_{n \in \mathbb{N}}, (\kappa^n_k)_{k \leq n})$ be a fully compatible, slightly pure, and uniform $d$-ary cloning system.
Then there exists an ergodic free p.m.p action such that $L^{\infty}(X, \mu) \rtimes \mathscr{T}_d(G_*)$ are McDuff factors; in particular $\mathscr{T}_d(G_*)$ is a McDuff group.
\end{introthm}

With cloning systems in mind, we can view the remaining results in this paper that pertain to the Higman-Thompson group $T_d$ and $V_d$ and they're corresponding von Neumann algebras, as a special case of a specific cloning system construction. 

The notion of proper proximality was introduced in \cite{boutonnet2021properly} by Boutonnet, Ioana and Peterson to study ``smallness'' phenomena at infinity for groups by generalizing the notion of bi-exactness to a large class of groups. Examples of properly proximal groups include bi-exact groups, non-elementary relatively hyperbolic groups, non-elementary convergence groups, all lattices in non compact semi-simple Lie groups, most graph products of groups etc. See for reference  \cite{boutonnet2021properly}, \cite{ding2021ProperProximality}, \cite{ding2022properly} \cite{ding2022upgrading}, \cite{ding2022first}, \cite{horbez2023proper}. In this article we show proper proximality for the Higman-Thompson groups $T_d$ and $V_d$.

\begin{introthm}[Corollary \ref{Cor:PropP}]\label{MainThmPropP}
    The Higman-Thompson groups $T_d$ and $V_d$ are properly proximal for every integer $d\geq 2$.
\end{introthm}

As  Popa's deformation/rigidity theory came into focus, there came an avalanche of unprecedented results establishing primeness for von Neumann algebras arising from groups, crossed products of groups acting on measure spaces, and equivalence relations. See for example \cite{popa2008superrigidity}, \cite{chifan2010bass}, \cite{vaes2013one}, \cite{boutonnet2012solid}, \cite{chifan2013structural}, \cite{hoff2016neumann}, \cite{dabrowski2016unbounded}, \cite{chifan2016primeness} \cite{chifan2018tensor}, \cite{drimbe2020prime}, \cite{isono2021tensor}, \cite{chifan2023tensor}, \cite{patchell2023primeness}, \cite{chifan2023rigidity}.  
Using soft analysis, we obtain primeness results for 
Higman-Thompson's groups $T_d$ and $V_d$ for integers $d\geq 2$.

\begin{introthm}[Corollary \ref{cor:Prime}]\label{MainThm:Prime}
     The group von Neumann algebras $L(T_d)$ and $L(V_d)$ of the Higman-Thompson group $T_d$ and $V_d$ are prime for every integer $d\geq 2$. 
\end{introthm}

In his recent work, \cite{patchell2023primeness} Patchell proved primness results for crossed product von Neumann algebras arising from generalized Bernoulli actions. We show $V_d$ and $T_d$ are concrete example and in particular, the following holds. 

\begin{introthm}[Corollary \ref{cor:Prime2}]\label{MainThm:Prime2}
    Let $(B, \tau)$ be a II$_1$ factor and $V_{[0,1/d)} < V_d$ the subgroup of elements fixing $[0,1/d)$. The natural action of $V_d \curvearrowright V_d/V_{[0,1/d)}$ gives rise to a prime II$_1$ factor, $M = B^{V_d/V_{[0,1/d)}} \rtimes V_d$ for every integer $d\geq 2$. Similarly, $N = B^{T_d/T_{[0,1/d)}} \rtimes T_d$ is a prime II$_1$ factor  for every integer $d\geq 2$. 
\end{introthm}

\subsection{Comments on the Proofs}

In Section \ref{sec:stab} we prove Theorem \ref{MainThrm:McDuff} by using stable measure equivalence relations (defined in the same section) and a result of Tucker-Drob from \cite{tucker2020invariant}. In particular, the same conditions on a cloning system that Bashwinger and Zaremsky showed in \cite{bashwinger2021neumann} are sufficient for $L(\mathscr{T}_d(G_*))$ to be a McDuff II$_1$ factor are also sufficient for $\mathscr{T}_d(G_*)$ to be a McDuff group. However, we should remark that in general $G$ being a countable ICC McDuff group does not imply $L(G)$ is a McDuff II$_1$ factor, as was proved by Kida in \cite{kida2015stability}. He showed the Baumslag-Solitar groups are McDuff groups, while Fima showed the group von Neumann algebra is prime in \cite{fima2010}. Independently, Bashwinger showed that $F_d$ is a McDuff group using character rigidity in \cite{bashwinger2023neumann}. Our result provides an alternate proof to this fact and generalizes it to a natural class of Thompson-like groups that arise from cloning systems. 

We prove Theorem \ref{MainThmPropP} in Section \ref{sec:Primeness} by demonstrating that $T_d$ and $V_d$ both admit proper cocycle into a non-amenable representation and thus are properly proximal by a result of \cite{boutonnet2021properly}. The cocycle, described in Section \ref{sec:Cocycles}, is the natural generalization of the one Farley defined in \cite{farley2003proper}.

Theorem \ref{MainThm:Prime} is a special case of Theorem \ref{thm:prime}, proved in section 4. The proof leverages relative amenability against Popa's intertwining technique, by way of the Gaussian deformation associated to the quasi-regular representation. In particular, $T_d$ and $V_d$ admit cocycles into quasi-regular representations that are not weakly $\ell^2$. Therefore, Theorem \ref{thm:prime} encompasses primeness for a different class of groups than \cite{peterson2009} result for solidity of groups that admit a proper cocycle into a multiple of the left regular representation, by generalizing the type of representation. In contrast, \cite{chifanSinclair2013} and \cite{chifan2016primeness} obtain solidity and primeness results for groups admitting quasi-cocycles, and more generally arrays, into representations that are weakly $\ell^2$. Given the respective work of Jolissaint \cite{jolissaint1998central} and Picioroaga \cite{picioroaga2006inner} on showing $L(F_d)$ is McDuff for $d\geq 2$, we also know that $L(T_d)$ and $L(V_d)$ are prime but not solid, since solidity passes to subalgebras.

The inner amenability / non-inner amenability dichotomy of $F$ and $V$ persists into their von Neumann algebras and into larger families of Thompson-Like groups. The ``$F$-like'' groups $\mathscr{T}_d(G_*)$ Bashwinger and Zaremsky identified give rise to McDuff II$_1$ factors $L(\mathscr{T}_d(G_*))$, while Theorem \ref{MainThm:Prime} shows $L(V_d)$ is prime. This dichotomy continues into the group measure space setting where Theorem \ref{MainThrm:McDuff} shows the same ``$F$-like'' groups $\mathscr{T}_d(G_*)$ have an ergodic free p.m.p action such that $L^{\infty}(X, \mu) \rtimes \mathscr{T}_d(G_*)$ is McDuff and, in particular, has property Gamma. In contrast, Corollary \ref{cor:VnonGamma}, which follows from Lemma \ref{lem:notcoamen} and \cite{chifan2016inner}, tells us there exists a free strongly ergodic p.m.p action such that $L^{\infty}(X, \mu) \rtimes V_d$ does not have property Gamma. It would be interesting
to see if the ``$V$-like groups'' arising from cloning systems will follow this pattern.

\section*{Acknowledgements} The authors would like to thank I.~Chifan and S.~Kunnawalkam Elayavalli for their contributions to Theorem \ref{thm:prime}; T.~Sinclair and I.~Chifan for the insightful discussions and assistance fixing an earlier version of this paper; G.~Patchell for his assistance; E.~Bashwinger for his helpful comments, assistance with the diagrams, and for pointing out $L(V_d)$ is not solid; and to M.~Zaremsky for his helpful comments improving the readability of the manuscript. The first author was supported in part by NSF Grant DMS-2055155.

\section{Preliminaries}

\subsection{von Neumann Algebras}\label{subsec:vNA}
A unital $*$-subalgebra $M\subseteq \mathcal B(\mathcal H)$  is called a von Neumann algebra if it is closed in the Weak Operator Topology (WOT); i.e. if $\{x_\lambda\}_{\lambda\in \Lambda}$ is a net in $M$ and $x_0\in \mathcal B(\mathcal H)$ is such that for all $\xi,\eta\in \mathcal H$ $\langle (x_\lambda-x_0)\xi,\eta \rangle\to 0$, then it follows that $x_0\in M$.  Equivalently, von Neumann's bicommutant theorem show that $M$ is a von Neumann algebra if $M=M''$, where $M'=\{y\in \mathcal B(\mathcal H) : yx=xy,\forall x\in M\}$.  

A von Neumann Algebra $(M, \tau)$ is \textit{tracial} if it has a faithful normal tracial state $\tau$. The completion of $M$ with respect to the norm $\|x\|_2 = \sqrt{\tau (x^*x)}$ is denoted by $L^2(M)$. We assume $M$ to be separable unless noted otherwise and that $M \subset \mathcal B (L^2(M))$ is the standard representation. We denote $\mathscr{P}(M)$ the projections for $M$, $\mathcal{U}(M)$ the group of unitaries of $M$, $(M)_1 = \{x\in M | \|x\| \leq 1 \}$ the unit ball of $M$, and $\mathcal Z(M) = M'\cap M$ the center of $M$. 

Let $\mathcal Q\subseteq M$ be a von Neumann subalgebra. The basic construction $\langle M, e_{\mathcal Q}\rangle$ is denoted as the von Neumann subalgebra of $\mathcal B(L^2(M))$ generated by $M$ and the orthogonal projection $e_{\mathcal Q}$ from $L^2(M)$ onto $L^2({\mathcal Q})$. There is a semi-finite faithful trace on $\langle M,e_{\mathcal Q}\rangle$ given by $Tr(xe_{\mathcal Q}y)=\tau(xy)$ for every $x,y\in M$.  Denote by $\mathcal{N}_M(\mathcal Q) = \{u \in \mathcal{U}(M) | u\mathcal{Q}u^*=\mathcal Q \}$ the \textit{normalizer of} $\mathcal Q$ in $M$ and say the subalgebra $\mathcal Q \subset M$ is \textit{regular} if $\mathcal{N}_M(\mathcal Q)'' = M$.

Given a discrete countable group $G$, we consider $\ell^2(G)$, the space of square summable sequences indexed by $G$ with basis $\{ \delta_g\}_{g\in G}$; $G$ acts on $\ell^2(G)$ though the left-regular representation $\lambda:G\to \mathcal B(\ell^2(G))$, which is characterized by linearly extending $\lambda_g(\delta_h):=\delta_{gh}$ for all $g,h\in G$; $$L(G):=\overline{\{\sum_{g\in F} c_g\lambda_g: F\subseteq G \text{ finite},  c_g\in \mathbb{C} \}}^{WOT}= \mathbb{C}[\lambda(G)]''\subseteq \mathcal B(\ell^2(G))$$ is called the group von Neumann algebra associated to $G$. The linear functional, $\tau:L(G)\to\mathbb{C}$ defined by $\tau(x):=\langle x\delta_e,\delta_e \rangle$ is a normal, faithful, tracial state, simply called a trace. It is well-known that $L(G)$ is a II$_1$ factor precisely when $G$ is a non-trivial ICC group.

Given an action $G \curvearrowright^\alpha X$ by a countable discrete group on a standard probability space, the group measure space construction $$L^\infty(X, \mu) \rtimes G := \mathbb{C} \langle \pi_\alpha (L^\infty (X, \mu)), \lambda(G)  \rangle'' \subset \mathcal{B}(\ell^2(G) \bar\otimes L^2(X,\mu))$$ contains a copy of $L(G)$ and $L^\infty(X, \mu)$ with a twisted commutation relation defined by $\pi_\alpha(f) (\delta_g \otimes f_g) = \alpha_{g^{-1}}(f)\delta_g \otimes f_g$ for $f \in L^\infty(X,\mu), f_g \in L^2(X, \mu)$  and $\lambda_g(\delta_h \otimes f_h) = \delta_{gh}\otimes f_h.$ When $G$ is infinite and the action is probability measure preserving (p.m.p), essentially free, and ergodic, $L^\infty(X, \mu) \rtimes G$ is a II$_1$ factor, with trace defined by linearly extending $\tau (x) := \langle x(\delta_e \otimes 1), \delta_e \otimes 1 \rangle.$

For a countable equivalence relation $\mathcal{R} \subset X \times X,$ where $(X,\mu)$ is a probability space, there is an associated von Neumann algebra $L(\mathcal{R}).$  Let $[\mathcal{R}]$ be the full group of Borel automorphisms that preserve equivalence classes in $\mathcal{R}$. Then for each $g \in [\mathcal{R}]$ there is a corresponding unitary $u_g \in \mathcal{U}(L^2(\mathcal{R},m))$ defined by $[u_gf](x,y) =f(g^{-1}x,y)$. We also have that each $a \in L^\infty(X)$, is identified with the multiplication operator in $\mathcal{B}(L^2(\mathcal{R},m))$, which is defined by $[af](x,y) = a(x)f(x,y).$ Then we define the von Neumann algebra of the equivalence relation $\mathcal{R}$ as
$$L(\mathcal{R}) = \{ L^\infty(X), \{u_g: g \in [\mathcal{R}] \} \}'' \subset \mathcal{B}(L^2(\mathcal{R},m)).$$ 
$L(\mathcal{R})$ comes equipped with a faithful normal trace defined by $\tau (x) = \langle x1_D, 1_D \rangle$, where $1_D \in L^2(\mathcal R, m)$ is the characteristic function of $D = \{(x,x) : x \in X \}$. When the equivalence relation comes from an action by a countable discrete group, we have 
\[\mathcal{R} = \mathcal{R}_{G \curvearrowright (X,\mu)} := \{(x,gx) : x \in X, g \in G \}. \]
When the action is free, p.m.p, and ergodic, it is well-known that $L(\mathcal{R}_{G \curvearrowright X})$ is isomorphic to $L^\infty(X, \mu) \rtimes G.$ In Section \ref{sec:stab} we define stable equivalence relations and use this identification integrally. 

Let $M$ be a II$_1$ factor. A central sequence $(a_n)_{n \in \mathbb N}$ in $M$ is a sequence such that $\|xa_n - a_nx\|_2 \rightarrow 0$ as $n \rightarrow \infty$ for all $x \in M$. Two central sequences $(a_n)$ and $(b_n)$ in $M$ are equivalent if $\|a_n - b_n\|_2 \rightarrow 0$ as  $n \rightarrow \infty$. A central sequence is said to be trivial if it is equivalent to a sequence of scalars. A II$_1$ factor $M$ is said to have the \textit{McDuff property} if there exists two non-trivial, non-commuting central sequences in $M$. Equivalently, $M$ is McDuff if $M \cong M \bar \otimes \mathcal{R}_0$, where $\mathcal {R}_0$ is the hyperfinite II$_1$ factor. A II$_1$ factor $M$ has \emph{property Gamma} if it contains at least one non-trivial central sequence. For a II$_1$ factor with separable predual, \textit{not} having property Gamma is equivalent to being \emph{full} \cite[Corollary 3.8]{connes1974almost}. So for this article we use the negation of property Gamma and fullness interchangeably.

\subsubsection{Popa's Intertwining-by-bimodules}

We recall Popa's intertwining-by-bimodules technique. 

Let $N$ and $M$ be von Neumann algebras. Then given two $M$-$N$-bimodule $\mathcal{H}$ and $\mathcal{K}$, we say $\mathcal{H}$ is \textit{contained} in $\mathcal{K}$ and write $\mathcal{H} \subset \mathcal{K}$, if there exists an $M$-$N$-bimodular isometry $V: \mathcal{H} \rightarrow \mathcal{K}.$ We say $\mathcal{H}$ is \textit{weakly contained} in $\mathcal{K}$ and write $\mathcal{H} \prec \mathcal{K}$ if for all $\xi \in \mathcal{H}$, finite subsets $E \subset M$, $F \subset N$ and $\epsilon > 0$ there exits $n \in \mathbb N$ and $\{\eta_1, \dots \eta_n \}$ in $\mathcal{K}$ such that 

\[ | \langle x\xi y, \xi \rangle - \sum_{i=1}^n \langle x \eta_i y, \eta_i \rangle | < \epsilon\]
for all $x \in E$ and $y \in F$. If $\mathcal{H} \prec \mathcal{K}$ and $\mathcal{K} \prec \mathcal{H}$, then we say $\mathcal H$ and $\mathcal K$ are \textit{weakly equivalent} and write $ \mathcal{H} \sim \mathcal{K}$.

\begin {theorem}[{\cite[Theorem 2.1, Corollary 2.3]{Popa06StrongRigidityI}}] Let $( M,\tau)$ be a separable tracial von Neumann algebra and let $\mathcal P, \mathcal Q\subseteq  M$ be (not necessarily unital) von Neumann subalgebras. 
Then the following are equivalent:
\begin{enumerate}
\item There exist $ p\in  \mathscr P(\mathcal P), q\in  \mathscr P(\mathcal Q)$, a $\ast$-homomorphism $\theta:p \mathcal P p\rightarrow q\mathcal Q q$  and a partial isometry $0\neq v\in q M p$ such that $\theta(x)v=vx$, for all $x\in p \mathcal P p$.
\item For any group $\mathcal G\subset \mathcal U(\mathcal P)$ such that $\mathcal G''= \mathcal P$ there is no sequence $(u_n)_n\subset \mathcal G$ satisfying $\|E_{ \mathcal Q}(xu_ny)\|_2\rightarrow 0$, for all $x,y\in  M$.
\item There exist finitely many $x_i, y_i \in M$ and $C>0$ such that  $\sum_i\|E_{ \mathcal Q}(x_i u y_i)\|^2_2\geq C$ for all $u\in \mathcal U(\mathcal P)$.
\item (\cite[Proposition C.1]{vaes07})There exists $n \in \mathbb{N}$, a projection $p \in M_n(\mathbb{C}) \otimes \mathcal Q$, a *-homomorphism $\psi: \mathcal P \rightarrow p(M_n(\mathbb{C}) \otimes \mathcal Q)p$ and a non-zero partial isometry $v \in 1_{\mathcal P}(M_{1,n}(\mathbb{C}) \otimes M)p$ satisfying $xv=v\psi(x)$ for all $x \in \mathcal P$.
\end{enumerate}
\end{theorem} 
If one of these equivalent conditions holds, we write $\mathcal P\prec_M\mathcal Q$ and say that a \textit{corner} of $\mathcal P$ embeds in $\mathcal Q$ inside $M$.

\subsubsection{Relative Amenability} A tracial von Neumann algebra $(M,\tau)$ is \textit{amenable} if there exists a $M$-central state $\phi$ on $\mathcal{B}(L^2(M))$ such that $\phi|_M=\tau$. By $M$-central we mean, $\phi(xT)=\phi(Tx)$ for all $x\in M$ and $T \in \mathcal{B}(L^2(M))$. More generally, given a projection $p \in M$, and von Neumann subalgebras $\mathcal P \subset pMp$, $N \subset M$, we say $P$ is \textit{amenable relative} to $N$ inside $M$ if there exists a a $\mathcal P$-central state $\psi$ on $p\langle M, e_N \rangle p$ such that $\psi|_{pMp} = \tau$. When $\mathcal P=M$ and $p=1$, we say that $N$ is \textit{co-amenable} in $M$. If we further assume $N = \mathbb C$, then we recover the definition of amenability for $M$. We will need several other forms of relative amenability from \cite{ozawapopa2010cartanI} throughout this paper, which we state now for the case $p=1$. 

\begin{lemma}[Section 2.2 \cite{ozawapopa2010cartanI}] Let $P$ and $N$ be von Neumann subalgebras of a tracial von Neumanna algebra $(M, \tau)$. Then the following are equivalent:
    \begin{enumerate}
        \item $\mathcal P$ is amenable relative $N$ inside $M$
        \item there exists a $\mathcal P$-central state $\psi$ on $\langle M, e_N \rangle$ such that $\psi$ is normal on $M$ and faithful on $\mathcal Z (\mathcal{P}' \cap M)$;
        \item $_ML^2(M)_{\mathcal P}$ is weakly contained in $_ML^2(M)\otimes_NL^2(M)_{\mathcal P}$. 
    \end{enumerate}
\end{lemma}

The following condition for relative amenability that we will need was also established in Section 2.2 of \cite{ozawapopa2010cartanI}, but appears in this form in \cite{ionanaCartan2015}. The idea is that given a net satisfying the conditions in the next lemma, one can construct a $\mathcal P$-central state that is normal on $M$ and faithful on $\mathcal Z (\mathcal{P}' \cap M)$.  

\begin{lemma}[Lemma 2.3 \cite{ionanaCartan2015}]\label{lem:relameSeq}
    Let $\mathcal P$ and $\mathcal Q$ be two von Neumann subalgebras of a tracial von Neumann algebra $(M,\tau)$. Assume that there is a $\mathcal Q$-$M$-bimodule $\mathcal K$ and a net $(\xi_i)_{i \in I}$ of elements $\xi_i$ in a multiple of $\mathcal H$ of $L^2(M) \otimes_Q \mathcal K$ such that:
    \begin{enumerate}
        \item $\limsup_i \|x\xi_i\|_2 \leq \|x \|_2$ for all $x \in M$;
        \item $\limsup_i \|\xi_i \|_2 > 0$;
        \item $\lim_i\|y\xi_i -\xi_i y \|_2=0$ for all $y \in \mathcal P$.
    \end{enumerate}

    Then, there exists a non-zero projection $p' \in \mathcal{Z}(\mathcal P'\cap M)$ such that $\mathcal{P}p'$ is amenable relative to $\mathcal Q$ inside $M$. 
\end{lemma}

\subsection{Cloning Systems}\label{section:CS}

In 1965 the Thompson Groups, $F, T,$ and $V$ were introduced by Richard Thompson. Successive subgroups of each other and countably discrete, $V$ is the group of piecewise linear bijections from the unit interval to itself, that map dyadic rationals to dyadic rationals. The elements are differentiable except at finitely many dyadic rationals and when the derivative does exist, it is a power of $2.$ An element of $T$ can be viewed as a piecewise linear homeomorphism, of $S^1$ after identifying the endpoints of $[0,1].$ An element of $F$, the most restrictive of Thompson's groups, must strictly be a piecewise linear homeomorphism on $[0,1].$  Cannon, Floyd, and Parry give a fundamental introduction to the Thompson Groups in \cite{cannon1996introductory} and we recall the generators of $F, T$ and $V$ below. 

The group $F$ is generated by the following elements $A$ and $B$, 

\begin{equation}\label{eq:Fgen}  A(x) = \begin{cases}  \frac{x}{2}, & 0 \leq x < \frac{1}{2} \\
    x-\frac{1}{4}, & \frac{1}{2} \leq x < \frac{3}{4} \\
    2x-1, & \frac{3}{4} \leq x < 1 
    \end{cases}  \quad
    B(x) = \begin{cases}  x, & 0 \leq x < \frac{1}{2} \\
    \frac{x}{2} + \frac{1}{4}, & \frac{1}{2} \leq x < \frac{3}{4} \\
    x-\frac{1}{8}, & \frac{3}{4} \leq x < \frac{7}{8} \\
    2x-1, & \frac{7}{8} \leq x < 1
    \end{cases}, 
\end{equation}
while $T$ is generated by $A,B,$ and $C$ and $V$ is generated by $A, B, C$ and $\pi_0,$ 

\begin{equation}\label{eq:Vgen}  C(x) = \begin{cases}  \frac{x}{2} + \frac{3}{4}, & 0 \leq x  < \frac{1}{2} \\
    2x- 1, & \frac{1}{2} \leq x < \frac{3}{4} \\
    x-\frac{1}{4}, & \frac{3}{4} \leq x < 1 
    \end{cases}  \quad
    \pi_0(x) =  \begin{cases}  \frac{x}{2} + \frac{1}{2}, & 0 \leq x  < \frac{1}{2} \\
    2x- 1, & \frac{1}{2} \leq x < \frac{3}{4} \\
    x, & \frac{3}{4} \leq x < 1 
    \end{cases}. 
\end{equation}

Early on Higman showed $T$ and $V$ are non-amenable in \cite{higman1974finitely}, but it is a persistent open problem to decide whether or not $F$ is amenable.  In \cite{HO2016} Haagerup and Olesen went a step further and showed $T$ and $V$ are non-inner amenable. On the other hand, Jolissaint showed $F$ to be inner amenable in \cite{jolissaint1997inner} and its group von Neumann algebra $L(F)$ to be McDuff in \cite{jolissaint1998central}.

There have been many generalizations of $F \subset T \subset V$ over the years, including Higmans's description of $V_d$, where he replaces dyadic rationals with $d$-adic rationals for any integer $d \geq 2$ in \cite{higman1974finitely}. Brown was technically the first to study the $d$-adic version of $F$ and $T$ in \cite{brown1987finiteness}. However, he wrote that it was ``the simply obvious generalization'' of Higman's work on $V_d$. Thus, the family of groups $F_d \subset T_d \subset V_d$ are now collectively referred to as the Higman-Thompson groups.

Our focus now is to provide the framework for a general $d$-ary cloning system as given by Skipper and Zaremsky in \cite{skipper2021almost}, building from the $d=2$ case that Witzel and Zaremsky developed in \cite{witzel2018thompson}.  These approaches are built off of Brin's treatment of the ``braided Thompson group $V$,'' or $bV,$ in \cite{brin2007algebra}, and the generalization of $bV$ to $bV_d$ given by Aroca and Cumplido \cite{BraidedVdAA}. We note $bV$ was also described independently by Dehornoy in \cite{dehornoy2006group}.  The starting point is viewing $F_d, T_d$ and $V_d$ as maps between $d$-ary trees. It is from this perspective that Thompson-like groups from cloning systems arise.

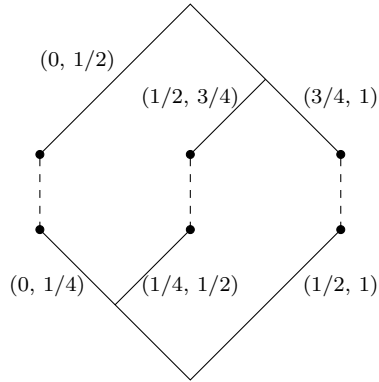
\begin{figure}
    \centering
    \begin{tikzpicture}
\draw (0,0) --  (2,2) -- (4,0)  (3,1) -- (2,0);
\node at (.5,1.25) {\tiny (0, 1/2)};
\node at (2,.75) {\tiny (1/2, 3/4)};
\node at (4,.75) {\tiny (3/4, 1)};
    \filldraw (0,0) circle (1.5pt);
    \filldraw (2,0) circle (1.5pt);
    \filldraw (4,0) circle (1.5pt);
    \draw[dashed] (0,-1) -- (0,0);
    \draw[dashed] (2,-1) -- (2,0);
    \draw[dashed] (4,-1) -- (4,0);
\draw (0,-1) -- (2,-3) -- (4, -1) (1,-2) -- (2, -1);
\node at (.1,-1.75) {\tiny (0, 1/4)};
\node at (2,-1.75) {\tiny (1/4, 1/2)};
\node at (4,-1.75) {\tiny (1/2, 1)};
    \filldraw (0,-1) circle (1.5pt);
    \filldraw (2,-1) circle (1.5pt);
    \filldraw (4,-1) circle (1.5pt);
\end{tikzpicture}
    \caption{Above is one of the generators, $A(x),$ of $F$ from equation \ref{eq:Fgen}. The binary trees represent how the unit interval is mapped homeomorphically by elements of $F.$}
    \label{fig:gen_F}
\end{figure}    

Figure \ref{fig:gen_F} depicts $A(x)$ as a map between binary trees indicating where each branch represents a subdivision of the unit interval.  Figure \ref{fig:Vgen} (b) depicts the element $\pi_0.$ The top tree is labeled $T$ and the bottom tree is labeled $U$ and a permutation is now necessary to describe how the trees are mapped to each other. Thus, an element of $V$ is represented by the triple $(T, \sigma, U).$ However, there is not a unique triple representing elements of $F$ or $V$ from this perspective, (see Figure \ref{fig:Vgen} (a) and (c) for two other triples representing $\pi_0$) so an equivalence relation must be defined. Intuitively, two triples will be in the same equivalence class when one element can be obtained by adding finitely many $d$-ary carets to the end of the trees (cloning) in a way compatible with the permutation, in the sense that no new piecewise breaks are introduced that would alter the original function being represented. The exact equivalence class and group operation are given after definition \ref{def:cloning}. 

A Thompson-like group from a $d$-ary cloning system comes from taking a general sequence of groups $G_n$, that map into $S_n$ in a compatible way. We now define a $d$-ary cloning system as given in \cite{skipper2021almost}. The 2-ary version was defined first in \cite{witzel2018thompson}.


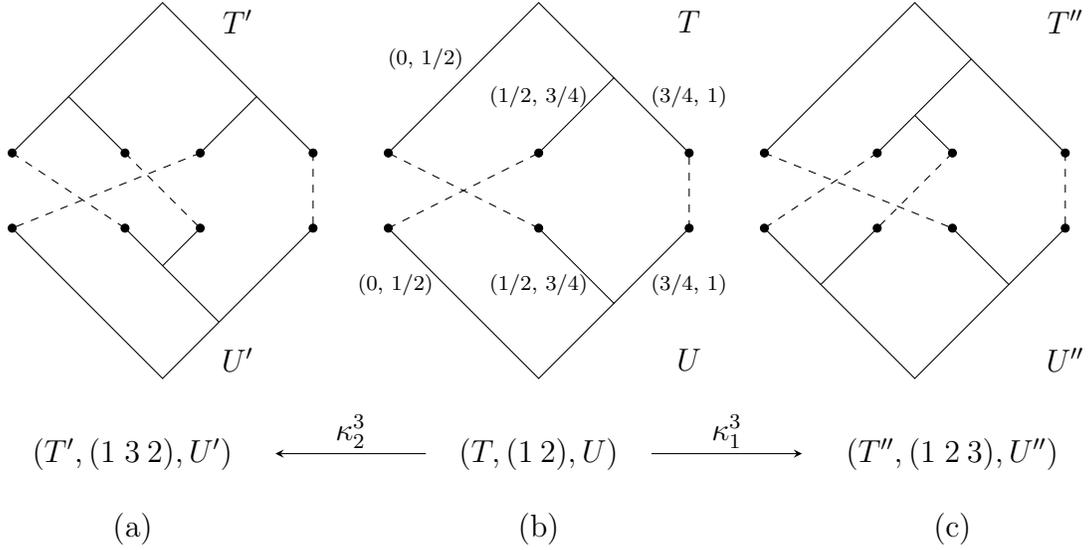
\begin{figure}
    \begin{tikzpicture}   
\draw (0,0) --  (2,2) -- (4,0)  (3,1) -- (2,0);
\node at (.5,1.25) {\tiny (0, 1/2)};
\node at (2,.75) {\tiny (1/2, 3/4)};
\node at (4,.75) {\tiny (3/4, 1)};
    \filldraw (0,0) circle (1.5pt);
    \filldraw (2,0) circle (1.5pt);
    \filldraw (4,0) circle (1.5pt);
\draw[dashed] (2,-1) -- (0,0);
\draw[dashed] (0,-1) -- (2,0);
\draw[dashed] (4,-1) -- (4,0);    
    \filldraw (0,-1) circle (1.5pt);
    \filldraw (2,-1) circle (1.5pt);
    \filldraw (4,-1) circle (1.5pt);
\draw (0,-1) -- (2,-3) -- (4, -1) (3,-2) -- (2, -1);

\node at (.1,-1.75) {\tiny (0, 1/2)};
\node at (2,-1.75) {\tiny (1/2, 3/4)};
\node at (4,-1.75) {\tiny (3/4, 1)};


\draw (5,0) -- (7,2) -- (9,0);
\draw (7.75, 1.25) -- (6.5,0);
\draw (7.5,0) -- (7,0.5);

\filldraw (5,0) circle (1.5pt);
\filldraw (6.5,0) circle (1.5pt);
\filldraw (7.5,0) circle (1.5pt);
\filldraw (9,0) circle (1.5pt);

\draw[dashed] (6.5,0) -- (5,-1);
\draw[dashed] (5,0) -- (7.5,-1);
\draw[dashed] (7.5,0) -- (6.5,-1);
\draw[dashed] (9,0) -- (9,-1);

\filldraw (5,-1) circle (1.5pt);
\filldraw (6.5,-1) circle (1.5pt);
\filldraw (7.5,-1) circle (1.5pt);
\filldraw (9,-1) circle (1.5pt);

\draw (5,-1) -- (7,-3) -- (9,-1);
\draw (6.5,-1) -- (5.75,-1.75);
\draw (7.5,-1) -- (8.25,-1.75);


\draw (-1,0) -- (-3,2) -- (-5,0);
\draw (-2.5,0) -- (-1.75,0.75);
\draw (-3.5,0) -- (-4.25,0.75);

\filldraw (-1,0) circle (1.5pt);
\filldraw (-2.5,0) circle (1.5pt);
\filldraw (-3.5,0) circle (1.5pt);
\filldraw (-5,0) circle (1.5pt);

\draw[dashed] (-5,0) -- (-3.5,-1);
\draw[dashed] (-3.5,0) -- (-2.5,-1);
\draw[dashed] (-2.5,0) -- (-5,-1);
\draw[dashed] (-1,0) -- (-1,-1);

\filldraw (-1,-1) circle (1.5pt);
\filldraw (-2.5,-1) circle (1.5pt);
\filldraw (-3.5,-1) circle (1.5pt);
\filldraw (-5,-1) circle (1.5pt);

\draw (-1,-1) -- (-3,-3) -- (-5,-1);
\draw (-3.5,-1) -- (-2.25,-2.25);
\draw (-2.5,-1) -- (-3,-1.5);


\node at (4,1.75) {$T$};
\node at (4,-2.75) {$U$};
\node at (9,-2.75) {$U''$};
\node at (9,1.75) {$T''$};
\node at (-2,1.75) {$T'$};
\node at (-2,-2.75) {$U'$};
\node at (2,-4) {$(T, (1 \: 2), U)$};
\node at (-3.4,-4) {$(T', (1 \; 3 \: 2), U')$};
\node at (7.5,-4) {$(T'', (1 \; 2 \: 3), U'')$};
\node at (2, -5) {(b)};
\node at (-3.4, -5) {(a)};
\node at (7.5, -5) {(c)};
\draw[-stealth] (3.5, -4) -- (5.5, -4);
\draw[stealth-] (-1.5, -4) -- (.5, -4);
\node at (4.5, -3.7) {$\kappa_1^3$};
\node at (-0.5, -3.7) {$\kappa_2^3$};

\end{tikzpicture}
    \caption{In (b), $\pi_0$, one of the four generators of $V$ from equation \ref{eq:Vgen} is depicted. Picture (a) and (c) are two different expansions of  $\pi_0.$ All three triples are different representations of the same element, $[T, (1 \; 2), U]$. To determine the permutations, the leaves are labeled left to right and the mapping reads from the bottom tree to the top tree.}
    \label{fig:Vgen}
\end{figure}


\begin{definition}[\cite{skipper2021almost} $d$-ary cloning system]\label{def:cloning}
 Let $(G_n)_{n \in \mathbb{N}}$ be a sequence of groups and $d \geq 2,$ be an integer.  For each $n \in \mathbb{N},$ let $\rho_n: G_n \rightarrow S_n$ be a homomorphism to the symmetric groups. For each $1 \leq k \leq n,$ let $\kappa_k^n: G_n \rightarrow G_{n+d-1}$ be injective, but not necessarily a homomorphism. A $\kappa_k^n$ is called a $d$-ary cloning map.  Following the literature, write $\rho_n$ to the left of its input and $\kappa_k^n$ to the right of its input. Then the triple 
\[((G_n)_{n \in \mathbb{N}}, (\rho_n)_{n \in \mathbb{N}}, (\kappa_k^n)_{k \leq n} ))\]
is a \textit{d-ary cloning system} if it satisfies the following axioms:
\begin{enumerate}[(C1):]
    \item (Cloning a product) $(gh)\kappa^n_k = (g)\kappa^n_{\rho_n(h)k}(h)\kappa^n_k$
    \item (Product of clonings) $\kappa^n_l \circ \kappa^{n+d-1}_k = \kappa^n_k \circ \kappa^{n+d-1}_{l+d-1}$
    \item (Compatibility) $\rho_{n+d-1}((g)\kappa^n_k)(i) = (\rho_n(g))\zeta^n_k(i)$ for all $i \neq k, k+1, \dots k+d-1$
\end{enumerate}
We always have $1 \leq k < l \leq n$ and $g,h \in G_n$ and $\zeta^n_k$ denotes the standard $d$-ary cloning maps for the symmetric groups, explained in more detail in \cite[Example 2.2.]{skipper2021almost},  and defined below in example \ref{ex:symm} for completeness. 
\end{definition}

We can now describe the group structure on the Thompson-Like groups $\mathscr{T}_d(G_*)$, which as Bashwinger puts it, can be viewed as a ``Thompson-esque limit'' of a sequence of groups, in contrast to a direct limit. 

Let $T$ and $U$ be $d$-ary trees, both with $n$ leaves and $g$ an element of $G_n.$ The triple $(T, g, U)$ represents an element of $\mathscr{T}_d(G_*).$ To define the equivalence classes for these triples, we need an \textit{expansion} of $(T, g, U),$ which is a triple $(T', (g)\kappa_k^n, U')$ (See figure \ref{fig:Vgen}). Adding a $d$-ary caret to the k$^{th}$ leaf of $U$ produces $U'$ while $T'$ is the result of adding a $d$-ary caret to the $\rho_n(g)(k)$-th leaf. The opposite of an expansion is called a \textit{reduction}.  Inductively extend expansions and reductions to be any finite iteration of either respective process. If one triple $(T, g, U)$ can be obtained from another triple via a finite sequence of expansions and reductions they are in the same equivalence class.  Thus, the elements of $\mathscr{T}_d(G_*)$ are the equivalence classes $[T, g, U].$

The group operation can only be carried out by choosing representatives $[T, g, U]$ and $[V, h, W]$ where $U=V.$  This is possible because two trees can always be expanded into each other by adding carets to the mismatched leaves until they agree.  Then the group operation is similar to the composition of elements in the braid groups, in the sense that the elements are stacked on top of each other and composed.
\[[T, g, U][U, h, W] := [T, gh, W]\]
The identity element comes from two copies of any $d$-ary tree, $[T, 1, T].$  The inverse operation is the natural inversion of the pictures and is given by 
\[[T, g, U]^{-1} = [U, g^{-1}, T].\]
We encourage reading the foundational and detailed work on cloning systems and Thompson-Like groups in \cite{UserGuideCloningZar},\cite{witzel2018thompson}, \cite{skipper2021almost},  \cite{bashwinger2021neumann}, and \cite{bashwinger2023neumann}. 

\subsubsection{Examples of Cloning Systems}

A powerful aspect of cloning systems is the unifying framework it provides for several natural generalizations of Thompson's groups that were introduced independently over the years. In addition to braided $V$ mentioned earlier, braided $F$ first introduced by Brady, Burillo, Clearly, and Stein \cite{brady2008pure} and the R\"{o}ver-Nekrashevych groups, which originated in the work of Nekrashevych in \cite{nekrashevych2004cuntz} and R\"{o}ver in \cite{rover1999constructing} can be described from a cloning system perspective. Throughout this paper, we review several of the examples from \cite{bashwinger2021neumann} because many of these same examples will satisfy Theorem \ref{MainThrm:McDuff}. More details on the group theoretic properties of the examples can be found in \cite{witzel2018thompson} and \cite{skipper2021almost}. 

\begin{example}\label{ex:symm}
To recover $V_d$ from a cloning system, take $(G_n)_{n \in \mathbb{N}} = (S_n)_{n \in \mathbb{N}},$ $\rho_n$ to be the identity maps at each level, and the cloning maps to be the standard symmetric group cloning maps needed in the (C3), $\kappa_k^n = \zeta_k^n: S_n \rightarrow S_{n+d-1}$, which we now define.

\[ 
((\sigma)\zeta^n_k)(i) := \begin{cases} \sigma(i) & \text{if} \: i \leq k \: \text{and} \: \sigma (i) \leq \sigma (k) \\
\sigma(i)+d-1 & \text{if} \: i < k \: \text{and} \: \sigma (i) > \sigma (k) \\ 
\sigma(i-d+1) & \text{if} \: i > k+d-1 \: \text{and} \: \sigma (i-d+1) < \sigma (k) \\
\sigma(i-d+1)+d-1 & \text{if} \: i \geq k+d-1 \: \text{and} \: \sigma (i-d+1) \geq \sigma (k) \\
\sigma (k) +i -k &  \text{if} \: k < i < k+d-1
\end{cases}
\]
See \cite{skipper2021almost} Example 2.2 for the computations showing $\zeta^n_k$ satisfies the cloning system axioms in the $d=2$ case. To recover $F_d,$ take $G_n = \{1\}$ for all $n,$ which forces $\rho_n$ and $\kappa^n_k$ to be trivial too. Then in the notation of cloning systems, $F_d = \mathscr{T}_d(\{1\}) < V_d = \mathscr{T}_d(S_*).$ For $T_d$ sitting as a subgroup of $V_d$, the elements $[T, \sigma, U]$ are the restriction of $\sigma \in S_n$ to the cyclic subgroup $\langle (1 2 \dots n) \rangle < S_n.$ 

\end{example}

\begin{example}
    The R\"{o}ver-Nekrashevych groups were among the concrete examples of groups arising from  $d$-ary cloning systems. Generally they are thought of as a certain subgroup of the almost-automorphisms of a $d$-ary tree, but from the cloning system perspective, the family of groups that generate them comes from wreath products between permutation groups and automorphisms of trees. As these groups are outside the scope of this paper, we direct the interested reader to \cite[Section 2.2]{skipper2021almost} for explicit details of their construction. 
    In \cite{bashwinger2021neumann} it is shown directly that the R\"{o}ver-Nekrashevych group is ICC and gives rise to a type II$_1$ factor. However, it is suspected by Bashwinger and Zaremsky that the groups are non-inner amenable. For this paper, the interest in the R\"{o}ver-Nekrashevych group lies in them being ICC counter-examples to several parts of definition \ref{def:csProp} below. 
\end{example}

Before discussing more examples, we introduce several definitions that were used in \cite{bashwinger2021neumann} to investigate the group von Neumann algebra $L(\mathscr{T}_d(G_*))$ and which we will use in Section \ref{sec:SGA} to provide information on the group measure space von Neumann algebra of $\mathscr{T}_d(G_*).$

\subsubsection{von Neumann Algebras of Thompson-Like Groups}

Jolissant proved $F$ was inner amenable in \cite{jolissaint1997inner} and then that $L(F)$ is a McDuff II$_1$ factor in \cite{jolissaint1998central}, marking the first von Neumann algebraic results for Thompson's groups. 

Bashwinger and Zaremsky, and Bashwinger on his own more recently studied, the von Neumann algebras of Thompson-Like groups. Their approach leads to four conditions which are sufficient for $L(\mathscr{T}_d(G_*))$ to be McDuff II$_1$ factor.  In this article, we use three of these conditions so we provide the definitions and brief motivation. We direct the reader to \cite{bashwinger2021neumann} and \cite{bashwinger2023neumann} for more detailed explanations.

\begin{definition}\label{def:csProp}
A $d$-ary cloning system $((G_n)_{n \in \mathbb{N}}, (\rho_n)_{n \in \mathbb{N}}, (\kappa_k^n)_{k \leq n} ))$ is:
\begin{enumerate}
    \item \textit{Fully Compatible} if $\rho_{n+d-1}((g)\kappa^n_k)(i) = (\rho_n(g))\zeta^n_k(i)$ for all $1 \leq k \leq n$ and all $1 \leq i \leq n+d-1,$ including $i = k, k+1, \dots, k+d-1.$ 
    \item \textit{Pure} if $\rho_n(g)$ is the identity permutation on $\{1, \dots, n \}$ for all $g \in G_n$ and all $n \in \mathbb{N}.$ In other words, all $\rho_n$ are the trivial homomorphism.
    \item \textit{Slightly pure} if $\rho_n(g)(n) = n$ for all $n \in \mathbb{N}$ and $g \in G_n.$
    \item \textit{Uniform} if 
    \[ \kappa^n_k \circ \kappa^{n+d-1}_\ell = \kappa^n_k \circ \kappa^{n+d-1}_{\ell'} \]
    for all $1 \leq k \leq n$ and $\ell, \ell'$ satisfying $k \leq \ell \leq \ell' \leq k+d-1.$ 
\end{enumerate}
\end{definition}

Fully Compatible is a strengthening of $(C3)$ in Definition \ref{def:cloning}. In general, cloning a group element and then representing it in $S_n$ might not agree with mapping first into $S_n$ and then cloning with the standard symmetric group cloning map $\zeta^n_k$, when being applied to leaves higher than the location of the cloning. However, fully compatibility is a natural assumption and holds in most of the examples that were previously studied, except for some R\"{o}ver-Nekrashevych groups. 

Fully compatible cloning systems also are guaranteed to have an important subgroup $\mathscr{K}_d(G_*) \subset \mathscr{T}_d(G_*),$ which we now define.  In a fully compatible $d$-ary cloning system, if $g \in \ker (\rho_n)$ then $(g)\kappa^n_k \in \ker (\rho_{n+d-1})$ for any $1 \leq k \leq n.$  Therefore, when a triple is of the form $(T, g, T)$ for $g \in \ker (\rho_{n(T)}),$ and $n(T)$ the number of leaves of $T,$ then any expansion of the triple is of the same form. Thus, we can define the subgroup to be

\begin{equation}
    \mathscr{K}_d(G_*) := \{[T,g,T] : g \in \ker (\rho_{n(T)}) \}.
\end{equation}

Slightly pure is a powerful, albeit severe restriction. Once a group element is represented in $S_n,$ it is only allowed to permute the first $1, \dots, n-1$ leaves. In particular, the right-most leaf must remain fixed. Stronger still, but more natural, the pure condition is when the image of group elements under $\rho_n$ act trivially, so no leaves of the trees are permuted. As the name suggests, pure implies slightly pure, and $F_d$ is the natural example of a pure cloning system. The group $V_d$ notably fails both conditions and highlights the dichotomy between pure, or ``$F$-like'' groups and not slightly pure, or ``$V$-like'' groups.

The uniform condition is needed to get asymptotic commutativity inside $\mathscr{T}_d(G_*)$. Intuitively, it tells us that after cloning a leaf by adding a $d$-ary caret, carrying out a second cloning does not depend on which newly added leaf is cloned.

Let us now recall Bashwinger and Zaremsky's use of fully compatible and slightly pure to decompose $\mathscr{T}_d(G_*)$ as a semidirect product. In particular, the next lemma gives a useful surjection from fully compatible cloning systems onto $V_d$ and is needed to prove Theorem \ref{MainThrm:McDuff}. 

\begin{lemma}\cite[Lemma 3.2]{witzel2018thompson}\label{lem:surj} Given a fully compatible $d$-ary cloning system, there is a map $\pi:\mathscr{T}_d(G_*) \rightarrow V_d$ given by sending $[T, g, U]$ to $[T, \rho_n(g), U],$ where $g \in G_n.$ The kernel is $\mathscr{K}_d(G_*)$ and the image is a group $W_d$ with the property $F_d \leq W_d \leq V_d.$

\end{lemma}

Another tool we will need to prove Theorem \ref{MainThrm:McDuff} is a subgroup and semi-direct product decomposition from \cite{bashwinger2021neumann} which we include for completeness.  The subgroup of interest is $\hat V_d < V_d.$  This subgroup arises naturally during the cloning system construction that gives rise to $\mathscr{T}_d(S_*) = V_d,$ by restricting to $\hat S_n \leq S_n,$ where $\hat S_n$ is the subgroup that fixes $n,$ meaning $\hat S_n \cong S_{n-1}.$  From here, it can be seen that the standard cloning system on $(S_n)_{n \in \mathbb{N}}$ restricts to a slightly pure cloning system on $(\hat S_n)_{n \in \mathbb{N}}.$  Thus, $\mathscr{T}_d(\hat S_*) = \hat V_d$ is the subgroup of $V_d$ with elements $[T, \sigma, U],$ where $\sigma$ does not permute the right most leaf. With that in mind, there is a well defined homomorphism 
\[ \theta : \hat V_d \rightarrow \mathbb{Z} \]
sending $[T, \sigma, U]$ to $\delta_r(T)-\delta_r(U),$ where $\delta_r(T)$ is a positive integer corresponding to how many branches are connected to the right-most side of the tree. In other words, $\delta_r(T)$ can be thought of as the distance from the root to the right-most leaf. 
When a cloning system is fully compatible, we can use lemma \ref{lem:surj} to construct the subgroup needed for the semidirect product. Let $D_d \leq \hat V_d$ be the kernel of the homomorphism, $\theta$ defined above, and define 
\begin{equation}
    D_d(G_*) := \pi^{-1}(D_d).
\end{equation}\label{Dgroup}
Then $D_d(G_*)$ is the subgroup of elements $[T, g, U]$ such that $\delta_r (T) = \delta_r (U).$ We also have that $[F_d,F_d] \leq \pi(D_d(G_*)) \leq D_d.$

We can now write $\mathscr{T}_d(G_*) \cong D_d(G_*) \rtimes \mathbb{Z},$ where the map $\mathscr{T}_d(G_*) \to \mathbb{Z}$ is the composition of $\pi: \mathscr{T}_d(G_*) \to \hat V_d$ with $\theta : \hat V_d \to \mathbb{Z}.$ In particular, we have the split exact sequence 
\[1 \rightarrow D_d(G_*) \rightarrow D_d(G_*) \rtimes \mathbb{Z} \rightarrow \mathbb{Z} \rightarrow 1.\]

We now provide several more examples of cloning systems and indicate which properties they have. 

\begin{example}\label{ex:braid}

To construct $bV_d,$ from a $d$-ary cloning system, we start with the family of braid groups $(B_n)_{n \in \mathbb{N}}.$  An element $b \in B_n$ is oriented by labeling the strands from left to right. Then the cloning maps, $\kappa_k^n$ send $b \in B_n$ to the element in $B_{n+d-1}$, where the $k^{th}$ strand is replaced by $d$ parallel strands.  The maps $\rho_n: B_n \rightarrow S_n$ are the projections onto the symmetric groups that remember where the strand ends, but not the order of the braiding. From this perspective, we then have $bV_d = \mathscr{T}_d(B_*).$ The same maps, together with the restriction to the family of pure braid groups, $(PB_n)_{n \in \mathbb{N}}$, give rise to braided $F_d,$ written as $bF_d =\mathscr{T}_d(PB_*).$ These $d$-ary braided Higman-Thompson groups were first studied by Aroca and Cumplido in \cite{aroca2022new} and by Skipper and Wu in \cite{skipper2021finiteness}, where both papers consider more general constructions using wreath products. In \cite{bashwinger2021neumann}, Bashwinger and Zaremsky show 
how the cloning system $bF_d$ arises from is pure and uniform, while the cloning system $bV_d$ arises from is not even slightly pure. 

\end{example}

\begin{example}\label{ex:matrix}
    Two examples of families of upper triangular matrix groups lead to Thompson-like groups. They were first introduced for $d=2$ in \cite{witzel2018thompson}, then generalized to all $d$ in \cite{bashwinger2021neumann}, where the group von Neumann algebras were studied.  Let $(B_n(R))_n$ be the family of $n$-by-$n$ invertible upper triangular matrices over a countable and commutative ring $R.$ The $k$th cloning map embeds an $n$-by-$n$ matrix into a $n+d-1$-by-$n+d-1$ matrix by duplicating the $a_{k,k}$ entry $d-1$ times along the diagonal, where the newly introduced columns duplicate the $k$th column entries and introduce $0$s to keep the new matrix upper triangular. Section 4.3 of \cite{bashwinger2021neumann} contains an example of such a cloning.  Modding out by the center, $\bar{B}_n(R) := B_n(R)/Z_n(R),$ provides a family of groups such that $\mathscr{T}_d(\bar{B}_n(R) )$ is ICC. Abels' groups, $Ab_n$, first studied in \cite{abels1987finiteness} are a special case when $R = \mathbb{Z}[1/p]$ and the entries $a_{1,1}$ and $a_{n+1, n+1}$ are both 1. Thus, $Ab_n$ is a subgroup of $B_{n+1}(\mathbb{Z}[1/p])$. 
    Bashwinger and Zaremsky also point out how both of these cloning systems are pure and uniform.
\end{example}

\begin{example}\label{ex:endo}
Another example that was introduced by Tanushevski in \cite{tanushevski2016new}  and generalized by Bashwinger and Zaremsky to the $d > 2$ case in \cite{bashwinger2021neumann} comes from taking $d$ injective endomorphisms $\phi_i, 1\leq i \leq d,$ of a countable group, $G$. The family of groups needed for the cloning system is then the $n$-fold direct product of $G,$  $(\Pi^n(G))_n.$  The cloning maps, $\kappa_k^n: \Pi^n(G) \rightarrow \Pi^{n+d-1}(G)$ are defined by 
\[(g_1, \dots, g_n)\kappa^n_k = (g_1, \dots, g_{k-1}, \phi_1(g_k), \dots, \phi_d(g_k), g_{k+1}, \dots, g_n). \]
If $G$ is ICC, then $\mathscr{T}_d(\Pi^n(G)_*)$ also will be ICC. However, if $G$ is not ICC, $\mathscr{T}_d(\Pi^n(G)_*)$ will no longer be either, without adding an additional assumption. On the other hand, restricting to $\Psi^n(G) := \{1\} \times \Pi^{n-1}(G) \leq \Pi^n(G),$ we get another example, first introduced in \cite{bashwinger2021neumann}, where $\mathscr{T}_d(\Psi^n(G)_*)$ is ICC, without $G$ necessarily being ICC. Bashwinger and Zaremsky point out that both of these $d$-ary cloning systems are pure, but the endomorphisms must all be the identity for the $d$-ary cloning systems to be uniform.

\end{example}

\section{Stable Group Actions}\label{sec:stab}
\subsection{Stability}

As noted in Section \ref{subsec:vNA}, a free, ergodic, p.m.p action of a countable group $G$ on a measure space gives rise to countable measured equivalence relation, which in turn generates a von Neumann algebra isomorphic to $L^\infty (X, \mu) \rtimes G.$  

We say an equivalence relation $\mathcal{R}$ is $\textit{stable}$ if $\mathcal{R} \simeq \mathcal{R} \times \mathcal{R}_0$, where $\mathcal{R}_0$ is the hyper-finite equivalence relation. From a von Neumann algebraic perspective, it is well-known that given a free, ergodic, p.m.p action, $\mathcal{R}_{G \curvearrowright (X,\mu)}$ being stable implies $L(\mathcal{R}_{G \curvearrowright (X,\mu)}) \cong L^\infty (X, \mu) \rtimes G$ is a McDuff II$_1$ factor. It is then immediate that a group giving rise to a stable equivalence relation is a McDuff group following the definition of Deprez and Vaes in \cite{deprez2018inner}.

Robin Tucker-Drob gave several classes of stable groups, including any nontrivial countable subgroups of $H(\mathbb{R}),$ the group of piecewise projective homeomrphisms, in \cite{tucker2020invariant}. Given that $F$ is such a subgroup of $H(\mathbb{R}),$ it is stable and thus a McDuff group. As Tucker-Drob remarks, it was already implicit in the literature that $F$ is stable. Bashwinger recently showed the $F_d$ is a McDuff group using character rigidity in \cite{bashwinger2023neumann}. In Section \ref{sec:SGA} we use Tucker-Drob's work to show a much larger class of Thompson-Like groups $\mathscr{T}_d(G_*)$ are stable.  

This idea of stability can be framed in terms of measured equivalence relations, orbit equivalence of groups, measure equivalence of groups, and discrete measured groupoids.  These perspectives are in general not equivalent to each other. However, when the action of $G$ is free, ergodic, and p.m.p, they all coalesce.

To recall the history briefly, Feldman and Moore's famous result tells us that every countable discrete measured equivalence relation $\mathcal{R}$ can be induced by a group action on the measure space (though not uniquely) \cite{feldman1977ergodic}. When $G$ and $H$ are countable discrete groups acting on probability space, the actions 
$G \curvearrowright (X,\mu)$ and $H \curvearrowright (Y,\nu)$ are \textit{orbit equivalent}, written $G \sim_{OE} H,$ if there exists an isomorphism $\phi: X \rightarrow Y$ that is measure-preserving and $\phi (G x) = H \phi (x)$ for a.e. $x \in X.$  We say $G$ is \textit{stable} if the free ergodic p.m.p action, $G \curvearrowright (X,\mu)$ is orbit equivalent to $G \times A \curvearrowright (X \times Y,\mu \times \nu)$ for any amenable group $A$, written $G \sim_{OE} G \times A.$
 Two groups are \textit{measure equivalent}, when there exist essentially free, ergodic, p.m.p actions which are weakly orbit equivalent, written $G \sim_{ME} H$. Note this is an equivalent formulation of measure equivalence first proved by Furman in \cite{furman1999orbit}. Orbit equivalence always implies measure equivalence, but measure equivalence is in general weaker. 

Jones and Schmidt showed in \cite{jones1987asymptotically} that a group action is stable, if and only if there is a specific sort of asymptotically central sequence of equivalence class preserving Borel automorphisms. The precise definition can be found in the same foundational paper. This approach is very useful for showing a group is stable, see for example \cite{kida2015stability}. A discrete measured groupoid $G \ltimes (X,\mu)$ generalizes a discrete measured equivalence relation. Kida gives the definition and a nice explanation of this connection in the beginning of Section 3 of \cite{kidastable2015}.  In particular, $\mathcal{R}_{G \curvearrowright (X,\mu)} \cong G \ltimes (X,\mu)$ if and only if $G \curvearrowright (X,\mu)$ is free.

Amenable groups are of course stable and Kida showed in \cite{kidastable2015} that any discrete countable group with the Haagerup property and having infinite center is stable.  On the other hand, countable discrete groups with property $(T)$ can never be stable.

We now include the precise theorem of Tucker-Drob, which we use to show certain $d$-ary cloning systems give rise to Thompson-Like groups that are stable.

\begin{theorem} \cite[Theorem 18]{tucker2020invariant}\label{TD18}
Let $1 \to N \to G \to K \to 1$ be a short exact sequence of groups in which $K$ is amenable. Then $G$ is stable if the following holds: \\
(H6) $N$ is doubly asymptotically commutative; meaning, there exists sequences $(c_n)_{n \in \mathbb{N}}$ and $(d_n)_{n \in \mathbb{N}}$ in $N$ such that $c_nd_n \neq d_nc_n$ for all $n \in \mathbb{N},$ and each $h \in N$ commutes with both $c_n$ and $d_n$ for cofinitely many $n \in \mathbb{N}.$

\end{theorem}

\subsection{Stable Group Actions for Thompson-Like Groups}\label{sec:SGA}

In this section, we prove Theorem \ref{thmstb} by showing when $\mathscr{T}_d(G_*)$ satisfies Theorem \ref{TD18}, a result of Tucker-Drob. 

We begin by constructing a single type of asymptotically commuting sequence inside $D_d(G_*)$, which can then be specified into a pair of doubly asymptotically commuting sequences. The sequences come from approximating the identity map between two $d$-ary trees. 

To capture this notion, we need to use the natural labeling of the vertices of any $d$-ary tree by finite words of the alphabet $\{0, \dots, d-1 \}.$  In particular, the vertices on the right side of the tree will be labeled with $(d-1)^n$ where $n$ is the number of branch levels to the vertex.  Given a vertex $v,$ which can now be identified with it's $d$-ary word label, we say two trees, $T$ and $U$ \textit{agree away from v} if there exists a tree containing a leaf labeled by $v$ such that both $T$ and $U$ can be obtained by adding some $d$-ary tree to the leaf $v.$

\begin{lemma}\label{lem:ac1} Let $((G_n)_{n \in \mathbb{N}}, (\rho_n)_{n \in \mathbb{N}}, (\kappa^n_k)_{k \leq n})$ be a fully compatible, slightly pure, and uniform $d$-ary cloning system. Then there exists an asymptotically commuting sequence in the subgroup $D_d(G_*)$.

\begin{proof}
    Recall that elements of $D_d(G_*)$ are triples of the form $[T,g,U],$ where $\delta_r(T)=\delta_r(U).$  That is, the distance to the right-most leaf in $T$ and $U$ must be the same. Now for each $n \in \mathbb{N},$ let $T_n$ and $U_n$ be two trees that agree away from $(d-1)^n,$ with the same number of leaves (see figure \ref{fig:AC_seq}) and let $E_n$ be the subset of $D_d(G_*)$ of elements that can be represented minimally by some $[R,g,S]$ with $\delta_r(R)=\delta_r(S) = n.$ Then the natural labeling implies that the rightmost leaf of both $R$ and $S$ is labeled by $(d-1)^n$ and \cite[Lemma 5.3]{bashwinger2021neumann}, $[T_n,1,U_n]$ commutes with every element of $E_n.$ We now observe that $[T_n,1,U_n]$ will also commute with every element in $E_m$ for any $0 < m \leq n.$ To see this, note that if $[R,g,S]$ can be represented minimally by some $(R,g,S)$ with $\delta_r(R)=\delta_r(S) = m,$ then it can also be represented by some $(R', g', S')$ with $\delta_r(R')=\delta_r(S') = n$, where $(R', g', S')$ is an expansion of $(R,g,S)$ that agrees up to $(d-1)^m$ and is the identity map on the leaves and vertices between $(d-1)^m$  and $(d-1)^n.$  Now because $D_d(G_*) = \cup_{n\geq 0} E_n$, it follows that $([T_n,1,U_n])_{n \in \mathbb{N}}$ is an asymptotically commuting sequence in $D_d(G_*).$
\end{proof}

\end{lemma}

\begin{lemma}\label{lem:ac2} Under the assumptions of the previous lemma, $D_d(G_*)$ is doubly asymptotically commutative. 

\begin{proof}
    We can now construct two asymptotically commuting sequences $(a_n)_{n \in \mathbb{N}}$ and $(b_n)_{n \in \mathbb{N}}$ in $D_d(G_*)$ such that $a_nb_n \neq b_na_n$ for all $n \in \mathbb{N},$ by building off the previous lemma. 

    Following the the definition of $D_d(G_*)$ in equation \ref{Dgroup}, we can view $[F_d, F_d] \leq D_d(G_*).$  We then choose $a_n$ and $b_n$ in $[F_d, F_d]$ such that the support of $a_n$ is exactly $(1-1/d^n, 1),$ and the support of $b_n$ is exactly $(1-(d+1)/d^{n+1},1),$ following what was done in the proof of \cite[Lemma 2.3]{bashwinger2021neumann}. By support, we mean the portion of the unit interval that is not fixed by a self-homoeomprhism in $F_d.$ By staggering the support of $a_n$ and $b_n$ in this way, we guarantee they do not commute. 

    Now writing, $a_n = [T_n,1,U_n]$ and $b_n = [S_n,1,R_n],$ the choice of support tells us both pairs, $T_n, U_n$ and $S_n, R_n,$ will agree away from $(d-1)^{n-1}$ (See figure \ref{fig:AC_seq} for an example of an element of one of these sequences). Therefore, the two sequences $(a_n)_{n \in \mathbb{N}}$ and $(b_n)_{n \in \mathbb{N}}$ are of the same form as the sequence in Lemma \ref{lem:ac1} above. Hence, $D_d(G_*)$ is doubly asymptotically commutative. 
\end{proof}
    
\end{lemma}

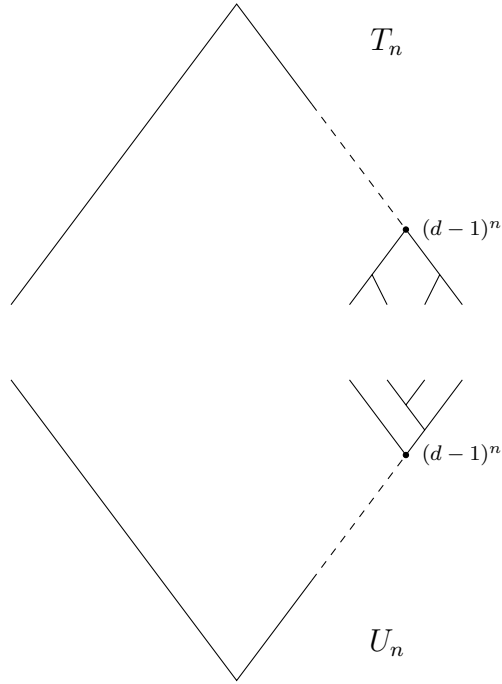
\begin{figure}
    \centering
    \begin{tikzpicture}
    \draw (-1,-1) -- (2,3) -- (3, 5/3) (4.25, 0) -- (5,-1) (4.25, 0) -- (3.5,-1) (3.8,-.6) -- (4,-1) (4.7,-.6) -- (4.5,-1);
    \draw[dashed] (3, 5/3) -- (4.25, 0);
\node at (5,0) {\tiny$(d-1)^n$};
\filldraw (4.25,0) circle (1pt);
    \draw (-1,-2) -- (2,-6) -- (3,-14/3) (4.25,-3) -- (5,-2) (4.25, -3) -- (3.5,-2) (4.5, -8/3) -- (4, -2) (17/4, -14/6) -- (4.5, -2);
    \draw[dashed] (3, -14/3) -- (4.25, -3);
    \node at (5,-3) {\tiny$(d-1)^n$};
\filldraw (4.25,-3) circle (1pt);
\node at (4,2.5) {$T_n$};
\node at (4,-5.5) {$U_n$};
\end{tikzpicture}
    \caption{This is an example of an element of the sequence $[T_n,1,U_n]$ from lemma \ref{lem:ac2}. Here $d=2$ and $\delta_r(T_n) = \delta_r(U_n) = n+2.$  Most of the element is the identity because $T_n$ and $U_n$ agree away from $(d-1)^n.$}
    \label{fig:AC_seq}
\end{figure}

\begin{theorem}[Theorem \ref{MainThrm:McDuff}]\label{thmstb}

Let $((G_n)_{n \in \mathbb{N}}, (\rho_n)_{n \in \mathbb{N}}, (\kappa^n_k)_{k \leq n})$ be a fully compatible, slightly pure, and uniform $d$-ary cloning system. Then the group $\mathscr{T}_d(G_*)$ is stable. Moreover, there exists an ergodic free p.m.p action such that $L^{\infty}(X, \mu) \rtimes \mathscr{T}_d(G_*)$ is a McDuff factor.
\begin{proof}

Since the $d$-ary cloning system is fully compatible, Lemma \ref{lem:surj} and the discussion proceeding it implies we can write $\mathscr{T}_d(G_*) \cong D_d(G_*) \rtimes \mathbb{Z},$ where $D_d(G_*)$ is defined in equation \ref{Dgroup}. The goal then is to show the split exact sequence, 
\[1 \rightarrow D_d(G_*) \rightarrow D_d(G_*) \rtimes \mathbb{Z} \rightarrow \mathbb{Z} \rightarrow 1,\]
satisfies Theorem \ref{TD18}. Given that that the $d$-ary cloning system is assumed to be fully compatible, slightly pure and uniform, Lemma \ref{lem:ac2} implies $D_d(G_*)$ is doubly asymptotically commutative.  Hence, the conditions for Theorem \ref{TD18} are satisfied, so $\mathscr{T}_d(G_*)$ is stable. From our discussion in Section \ref{sec:stab} we conclude that $\mathscr{T}_d(G_*)$ is a McDuff group.
\end{proof}

\end{theorem}

All of the groups $\mathscr{T}_d(G_*)$ constructed in \cite{bashwinger2021neumann} by Bashwinger and Zaremsky where $L(\mathscr{T}_d(G_*))$ are McDuff factors are also McDuff groups, answering question 6.4 that Bashwinger and Zaremsky posed in the same paper. In fact, Theorem \ref{thmstb} applies to larger class of groups, given that a group need not be ICC to be stable. 

The Thompson-Like groups from cloning systems, such that $L(\mathscr{T}_d(G_*))$ is a McDuff factor and $\mathscr{T}_d(G_*)$ is also a McDuff group include the following. The Higman-Thompson group $F_d$, which was shown to be a McDuff group independently in \cite{bashwinger2023neumann},  $bF_d,$ from example \ref{ex:braid}, $\hat{V}_d$, from the example following \ref{lem:surj}, $\mathscr{T}_d(\bar{B_*}(R))$ and $ \mathscr{T}_d(AB_*(R))$ from example \ref{ex:matrix}, and $\mathscr{T}_d(\Pi^n(G)_*)$ where the endomorphisms are the identity and $G$ is ICC from example \ref{ex:endo} and $\mathscr{T}_d(\Psi^n(G)_*)$, where the endomorphism are still the identity but $G$ is now not necessarily ICC also from example \ref{ex:endo}.

Notably, the Thompson-Like groups from cloning systems where $L(\mathscr{T}_d(G_*))$ is \textit{not} a McDuff factor but $\mathscr{T}_d(G_*)$ is a McDuff group, include $\mathscr{T}_d(B_*(R))$ from example \ref{lem:surj}, where we are no longer quotienting out by the center of $B_n(R)$ and $\mathscr{T}_d(\Pi^n(G)_*)$ from example \ref{ex:endo}, where we no longer assume $G$ is ICC, but the endomorphism are still the identity. However, in both of these examples $\mathscr{T}_d(G_*)$ is not ICC, so $L(\mathscr{T}_d(G_*))$ is no longer a factor.

\section{Deformations Associated to Quasi-Regular Representations}\label{sec:defRigPrimenessQuasi}
This section is devoted to proving our main result, Theorem \ref{thm:prime}, which gives criteria for a group von Neumann algebra to be prime when the group admits a cocycle into a quasi-regular representation that is not necessarily weakly contained in a multiple of the left-regular representation. Theorem \ref{thm:prime} could therefore be understood as generalizing \cite{peterson2009} with regards to the types of representations admitting unbound cocycles, while \cite{chifanSinclair2013} and \cite{chifan2016primeness} generalizes the notation of a cocycles itself (to quasi-cocycles and arrays).

Throughout this section we are going to use the following terminologies. Let $\pi:G\rightarrow \mathcal{O}(\ell^2(G/H))$ be an orthogonal representation of a countable group $G$ with a subgroup $H\leq G$ and let $c:G\rightarrow \ell^2(G/H)$ be a 1-cocycle. We now recall the Gaussian construction and s-malleable deformation associated to a 1-cocycle into an orthogonal representation, which was first defined in  \cite[Section 3]{sinclair2011strong} and subsequently used in \cite[Section 3.3 ]{peterson2012cocycle}.

 For an orthogonal representation $(\pi, \mathcal H)$ of a countable discrete group $G$ on a real Hilbert space $\mathcal H$, we have the associated Gaussian action $G\curvearrowright^{\sigma^\pi}L^{\infty}(X^\pi,\mu^{\pi})$, where $L^{\infty}(X^{\pi},\mu^{\pi})$ is the abelian von Neumann algebra generated by the unitaries $\{ \omega(\xi):\xi\in\mathcal H\}$. Then the action is abstractly characterized by the following conditions: $(a)\: \omega(\xi+\eta)=\omega(\xi)\omega(\eta)$, $(b)\:  \omega(\xi)^{\ast}=\omega(-\xi)$, $(c) \: \sigma^{\pi}_g\omega(\xi)=\omega(\pi(g)\xi)$, and $(d) \: \tau(\omega(\xi))=\exp(-\|\xi\|^2)$ for all $\xi \in \mathcal H$.

Letting $\tilde{M}=L^{\infty}(X^{\pi},\mu^{\pi})\rtimes_{\sigma^\pi} G $, we note that $L(G)\subseteq \tilde{M}$.  This is endowed with a trace $\tilde\tau (x)= \langle x (1\otimes \delta_e),1\otimes \delta_e ) \rangle$.  We can construct an s-malleable deformation for the inclusion $L(G)\subseteq \tilde M$, i.e.~a one-parameter family of automorphisms $\{\alpha_t\}_{t\in\mathbb R}$, given by:
$$ \alpha_t(u_g)= \omega(tc(g))u_g.$$

The variation of Popa's transversality inequality we will need appears in \cite{vaes2013one}, but is essentially contained in  \cite[Lemma 2.1]{popa2008superrigidity}.

\begin{lemma}[\cite{vaes2013one} Lemma 3.1 ]\label{lemma:trans} Assuming $\tilde M$ and $\alpha_t$ are as define above, then for all $x \in M$ and $t \in \mathbb{R}$ the following hold:
\[\|\alpha_t(x) - E_M(\alpha_t(x)) \|_2 \leq \|x - \alpha_t(x) \|_2 \leq \sqrt{2}\|\alpha_t(x) - E_M(\alpha_t(x)) \|_2.\]
    
\end{lemma}

Part of the ingenuity in Popa's well established intertwining technique, is it's interaction with the idea of mixing relative to a subalgebra.  

\begin{definition}
 Let $(M, \tau)$ be a tracial von Neumann algebra and $N\leq M$. Let $p\in \mathscr{P}(M)$ and $\mathcal Q\leq pMp.$ A $\mathcal Q$-$\mathcal Q$ bimodule ${}_{\mathcal Q}\mathcal{H}_{\mathcal Q}$ is \emph{mixing relative to $N$} if any net $(x_n)_{n \in I}$ in $(\mathcal Q)_1$ with $\|{E}_N(yx_nz)\|_2 \to 0$ for all $y, z \in M$ satisfies
\begin{align*}
    \lim_n \sup_{y \in (\mathcal Q)_1} |\langle x_n\xi y, \eta\rangle| = \lim_n \sup_{y \in (\mathcal Q)_1} |\langle y\xi x_n, \eta\rangle| = 0 \quad\text{ for all }\quad \xi, \eta \in \mathcal{H}.
\end{align*}
An $M$-$M$ bimodule ${}_M\mathcal{H}_M$ which is mixing relative to $\mathbb{C}$ is simply called \emph{mixing}.
\end{definition}

For us, the key insight is how mixing relative to a subalgebra in the absents of intertwining gives us control over relative commutants. This idea is captured in the following well known fact. 

\begin{lemma}[\cite{dSRHHS21} Lemma 6.3]\label{lemma:diffuse}
    Suppose tracial von Neumann algebras $A < M < \tilde M$ are such that $L^2(\tilde M) \ominus L^2(M)$ is mixing relative to $A$. Then for any $q \in \mathscr{P}(M)$ and $\mathcal Q \leq qMq$, either
    \begin{enumerate}
        \item $\mathcal Q \prec_M A,$ or
        \item $\mathcal Q' \cap q\tilde{M} q \subseteq qMq$.
    \end{enumerate}
\end{lemma}

The relative mixing property for bimodules arises naturally in the setting of the Gaussian construction. Indeed, let  $\pi:G\to \mathcal{O}(\mathcal{H})$ be representation of $G$ and $\mathcal{S}$ a family of subgroups of $G$. We say $\pi$ is \textit{mixing relative to} $\mathcal{S}$ if for any $\varepsilon>0$ and $\xi,\eta\in \mathcal{H}$, there exists $g_1,\ldots, g_n, h_1,\ldots, h_n\in G$ and $H_1,\ldots, H_n\in \mathcal{S}$ such that for any $g\in G\setminus \bigcup_{i=1}^ng_iH_ih_i$ we have $|\langle \pi(g)\xi, \eta \rangle|<\varepsilon$. It is readily checked that if $\mathcal{S}$ consists of solely the trivial group, then this is equivalent to mixing.  
Furthermore, the quasi-regular representation $\pi: G\to \mathcal{O}(\ell^2(G/H))$ is mixing relative to $\mathcal{S}=\{H\}$.  By \cite[Proposition 2.7]{boutonnet2012solid} (see also \cite{peterson2012cocycle}), we have the following.

\begin{lemma}\label{lem:relmix}
    Let $\tilde M$ be the Gaussian construction coming from the quasi-regular representation as defined above. Then $L^2(\tilde M)\ominus L^2(M)$ is mixing relative to $L(H)$.
\end{lemma}

\begin{lemma}\label{uniform-convergence-on-factors}
    Let $\pi:G\rightarrow \mathcal{O}(\ell^2(G/H))$ be the quasi-regular representation of a countable ICC group $G$ with respect to a subgroup $H < G$ and let $c:G\rightarrow \ell^2(G/H)$ be a 1-cocycle. Set $M:=L(G)$ and let $\alpha_t$ be the associated Gaussian deformation for $M$ in $\tilde M$, as defined above. 
     Let $\mathcal P, \mathcal Q\subseteq M$ be two diffuse factors such that $M=\mathcal P\Bar{\otimes}\mathcal Q$. Assume $\mathcal P$ is not amenable relative to $L(H)$ in $M$. Then $\alpha_t\rightarrow Id$ uniformly on $(\mathcal Q)_1$.
\end{lemma}

\begin{proof}
     Suppose for a contradiction that the convergence $\alpha_t\to \operatorname{Id}$ is not uniform on $ (\mathcal Q)_1$. Then there exists a sequence $(u_n)_{n \in \mathbb{N}} \in \mathcal{U}(\mathcal Q)$, $t_n > 0$ in $\mathbb{R}$, and $\kappa >0$ such that $\|\alpha_{t_n}(u_n) \|_2 > \kappa/ \sqrt{2}$ for all $n$ in $\mathbb{N}$. Setting $\delta_t(q):=\alpha_t(q)-E_{M}( \alpha_t(q))$ for any $q \in (\mathcal Q)_1$ and using Popa's transversality estimate above, we also have $\|\delta_{t_n}(u_n) \|_2 > \kappa$ for all $n$ in $\mathbb{N}$. Since $\mathcal P$ is not amenable relative to $L(H)$ in $M$, we also know that $\mathcal P$ is not amenable relative to $L(H)$ in $\tilde M$.  To obtain a contradiction, we claim $\xi_n : =\frac{\delta_{t_n}(u_n)}{\|\delta_{t_n}(u_n)\|_2}$ is a sequence that satisfies the conditions for relative amenability in Lemma \ref{lem:relameSeq}. First, given that $(\xi_n)$ is in $\tilde M < \langle \tilde M, e_{L(H)} \rangle$ for all $n$, we can take the $L(H)$-$\tilde M$-bimodule $\mathcal K$ to be $L^2(\tilde M)$ and view $(\xi_n)$ as a sequence in $L^2\langle \tilde M, e_{L(H)} \rangle$.
     
     Condition (1) of Lemma \ref{lem:relameSeq} is satisfied since the Cauchy-Schwarz inequality and $(\xi_n)$ being a sequence of norm one vectors implies $\limsup_n\|x\xi_n\|_2 \leq \|x\|_2$ for all $x \in M$. Condition (2) is immediate from the definition of $(\xi_n)$. For condition (3), we use a standard estimate that appears, for example, in \cite[Lemma 3.1]{hoff2016neumann}. Using $\delta_{t_n}(u_n)$ again, we see for any $y \in \mathcal P$,  

     \begin{align*}
 \|\delta_{t_n}(u_n)y-y\delta_{t_n}(u_n)\|_2
    = & \|(1- E_M)([\alpha_{t_n}(u_n),y])\|_2 \\
    \leq & \|[\alpha_{t_n}(u_n),y] \|_2 \\
    \leq & \| \alpha_{t_n}(u_n)y - \alpha_{t_n}(u_n)\alpha_{t_n}(y)\|_2 + \|[\alpha_{t_n}(u_n), \alpha_{t_n}(y)] \|_2 \\
    & + \| \alpha_{t_n}(y)\alpha_{t_n}(u_n) - y\alpha_{t_n}(u_n)\|_2 \\
     \leq & 2\|q\|\|\alpha_{t_n}(y)-y\|_2 +  \|[u_n, y] \|_2\\
    \leq & 2 \|\alpha_{t_n}(y)-y\|_2,
\end{align*}

where in the second to last line we use the fact that $\mathcal P = \mathcal Q' \cap M$. Now taking limits, we see that for all $y \in P$, 

\[\lim_{n\rightarrow \infty}\|\xi_ny -y\xi_n \|_2 \leq 2\lim_{n\rightarrow \infty}  \|\alpha_{t_n}(y)-y\|_2 = 0. \]

Therefore, the conditions of Lemma \ref{lem:relameSeq} are satisfied, so there exists a non-zero projection $p' \in \mathcal{Z}(\mathcal P' \cap M)$ such that $\mathcal Pp'$ is amenable relative to $L(H)$ inside $\tilde M$. However, we also know that by Lemma \ref{lem:relmix}, $L^2(\tilde M) \ominus L^2(M)$ is mixing relative to $L(H)$. Moreover, the non-relative amenability assumption also implies $\mathcal P \not \prec_{\tilde M} L(H)$. Therefore, Lemma \ref{lemma:diffuse} implies $\mathcal P'\cap \tilde M \subset M$, so in particular, $\mathcal P'\cap \tilde M = \mathcal P'\cap M = \mathcal Q$ and $\mathcal{Z}(\mathcal P' \cap \tilde M) = \mathbb{C}$. So in fact, $p' = 1$, which is a contradiction. Therefore, $\alpha_t \rightarrow Id$ uniformly on $(\mathcal Q)_1$, as $t \rightarrow 0$.
    \end{proof}

In light of the previous lemma, if we know that both $\mathcal P$ and $\mathcal Q$ are not amenable relative to $L(H)$ in $M$, then $\alpha_t\rightarrow Id$ uniformly on the subgroup $G =\mathcal{U}(\mathcal P)\mathcal{U}(\mathcal Q)$ of $\mathcal{U}(M)$, and $G'' = M$. (See for example the proof of \cite[Theorem 3.8]{dSRHHS21}). Therefore, a standard convexity argument  will upgrade the uniform convergence of $\alpha_t$ to all of $(M)_1$. 

\begin{lemma}\label{lemma:convex}
    If $\alpha_t\rightarrow Id$ uniformly on $\mathscr G$ for some $\mathscr G\leq \mathcal U(M)$ with $\mathscr G''=M$, then $\alpha_t\rightarrow Id$ uniformly on $(M)_1$.
\end{lemma}
\begin{proof}
    Let $K:=\overline{\operatorname{Co}\{\alpha_t(u)u^{\ast}: u\in \mathscr G\}}$. From the uniform convexity of Hilbert space, we have there exists a unique element of minimal norm, i.e., there exists $\xi_0\in K$ such that $\|\xi_0\|_2$ is minimal. Observe that $\alpha_t(v)\xi_0v^{\ast}\in K$ for all $v\in \mathscr G$. Indeed, for all $\xi=\sum_i \lambda_i\alpha_t(u_i)u_i^{\ast}$, we have $\alpha_t(v)\xi v^{\ast}=\sum_i\lambda_i\alpha_t(vu_i)(vu_i)^{\ast}$. Since $\|\alpha_t(v)\xi_0v^{\ast}\|_2=\|\xi_0\|_2$, by uniqueness, we have $\alpha_t(v)\xi_0v^{\ast}=\xi_0$ for all $v\in \mathscr G$. Since every $x \in M$ is a linear combination of four unitaries, we get $\alpha_t(x)\xi_0=\xi_0x$ for all $x\in M$. Let $\epsilon >0$ and observe that if $\|\alpha_t(u)-u\|_2\leq \epsilon $, then $\|\alpha_t(u)u^{\ast}-1\|_2\leq \epsilon$. From polar decomposition this further implies that $\alpha_t(x)v_0=v_0x$ for all $x\in M$ and for $v_0$ satisfying $\|v_0-1\|_2<\epsilon$. The following computation then concludes the result;
    $$\|\alpha_t(x)-x\|_2\leq \| \alpha_t(x)v_0+\alpha_t(x)(1-v_0)-v_0x-(1-v_0)x\|_2<2\epsilon.$$
\end{proof}

\begin{lemma}\label{bounded-cocycle}
   If $\alpha_t\rightarrow Id$ on $(M)_1=(L(G))_1$, then the 1-cocycle $c$ is bounded.
\end{lemma}

\begin{proof}
    Fix $\epsilon >0$, and choose $t$ such that $\|\alpha_t(u_g)-u_g\|_2^2\leq \epsilon^2 $ for all $g \in G$. Then for all $g\in G$,
  \begin{align*}
        & 2-2\tau(\alpha_t(u_g)u_g^{\ast})\leq \epsilon^2\\
        \Rightarrow & 2-\epsilon^2\leq 2Re(\tau(\omega(tc(g))))\\
        \Rightarrow & 2-\epsilon^2\leq 2e^{-\|tc(g)\|^2}\\
        \Rightarrow & \|c(g)\|^2\leq \frac{\ln{\frac{2-\epsilon^2}{2}}}{t}.
    \end{align*}
\end{proof}

Recall that a von Neumann subalgebra $N \subset M$ is said to be \textit{weakly bicentralized} in $M$ if $_ML^2\langle M, e_N \rangle_M \prec {_ML^2(M)_M}$, which by \cite[Lemma 3.4]{BannonFull2020} is always satisfied when $N = (N' \cap M^\omega)' \cap M$ for some ultrafilter $\omega \in  \beta\mathbb{N} \setminus \mathbb{N}$. Then the next lemma holds when $\mathcal P$ is amenable relative to $N$ inside $M$.

\begin{lemma}[c.f.~\cite{isono2021tensor} Lemma 5.2]\label{lemma:isono}
    Let $M$ be II$_1$ factor. Suppose $M \cong \mathcal P \bar\otimes \mathcal Q,$ $\mathcal P$ is a full factor, $N \subset M$ is II$_1$ subfactor that is weakly bicentralized in $M$, and $_{\mathcal P}L^2(\mathcal P)_{\mathcal P} \prec  {_{\mathcal P}L^2\langle M, e_N\rangle_{\mathcal P}}$ as $\mathcal P$-$\mathcal P$ bimodules. Then $\mathcal P \prec_M N.$ 
\end{lemma}
\begin{proof}
    This is implicit in the proof of \cite[Lemma 5.2]{isono2021tensor}. We reproduce the relevant portion for the sake of completeness. 
    
    Given that $N$ is weakly bicentralized in $M$ and $_{\mathcal P}L^2(M)_{\mathcal P}$ is a multiple of $_{\mathcal P}L^2(\mathcal P)_{\mathcal P}$, we have that $_{\mathcal P}L^2(\langle M, e_N\rangle )_{\mathcal P}\prec {_{\mathcal P}L^2(\mathcal P)_{\mathcal P}}$ as $\mathcal P$-$\mathcal P$ bimodules. This along with the assumption now implies that  $_{\mathcal P}L^2(\mathcal P)_{\mathcal P}$ is weakly equivalent to $_{\mathcal P}L^2(\langle M, e_N\rangle )_{\mathcal P}$. Since $\mathcal P$ is full, \cite[Proposition 3.2]{BannonFull2020} implies that $_{\mathcal P}L^2(\mathcal P)_{\mathcal P}$ is contained in $_{\mathcal P}L^2(\langle M,e_N\rangle)_{\mathcal P}$ as $\mathcal P$-$\mathcal P$ bimodules, and hence $\mathcal P\prec_M N$ by \cite{Popa06StrongRigidityI}.
\end{proof}

The following lemma is a slight modification of \cite[Theorem 6.16]{popavaes2008} whose proof follows almost exactly.


\begin{lemma}[c.f.~\cite{popavaes2008} Theorem 6.16]\label{lemma:intersec}
    Let $G$ be a countable ICC group and $H < G$ an ICC subgroup. Set $M:=L(G)$ and let $B \subset M$ be a von Neumann subalgebra. Suppose that the quasi-normalizer of $B$ in M is dense in M. 
    If $B \prec_M L(H)$, then $B \prec_M L(H \cap gHg^{-1})$ for all $g \in G$. Repeating the procedure, we get in particular that 
    \[ B \prec_M L\Bigg( \bigcap\limits_{i=1}^ng_iHg_i^{-1}\Bigg)\]
    for all $g_1, \dots, g_n \in G$.
\end{lemma}

We remark that \cite[Theorem 6.16]{popavaes2008} is stated for $G$ acting strictly outerly on a finite von Neumann algebra $(A,\tau)$. For our use, we need the conclusion to hold for $A=\mathbb{C}$ and $M = L(G)$, meaning when the action is not outer. However, strict outerness is only used to control where the central projections are located. Therefore, the proof of \cite[Theorem 6.16]{popavaes2008} goes through by replacing ``non-zero central projection'' with 1 when it occurs. We now include the proof for completeness. 

\begin{proof}
    First consider the set of self-intertwiners for $B$:
    \begin{align*}
        \mathcal{I}_n := \{a \in M_{1,n}(\mathbb{C}) \otimes M | & a \: \text{is a partial isometry for which there exists a possibly } \\
        & \text{non-unital, *-homomorphism} \: \alpha : B \rightarrow M_n(\mathbb{C}) \otimes B  \\
        & \text{s.t.} \: xa = a\alpha (x) \text{for all} \: x \in B \}.
    \end{align*}

For each $n \in \mathbb{N}$ the matrix coefficients of elements of $\mathcal{I}_n$ linearly span a subalgebra which is dense in $M$ because the quasi-normalizer of $B$ in $M$ is dense. 

Now to seek a contradiction, fix $g \in G$ and suppose that 
\[ B \prec_M L(H) \: \text{and} \: B \not \prec L(H \cap gHg^{-1}). \]

Then take a non-zero partial isometry $w \in M_{1,m}(\mathbb{C}) \otimes M$ and a, possibly non-unital, *-homomorphism $\theta: B \rightarrow M_m (\mathbb{C}) \otimes L(H)$ such that $xw = w \theta(x)$ for all $x \in B$. We claim  that there exists $w$ so that the smallest projection $p \in M$ satisfying $w^*w \leq 1 \otimes p$ can be assumed to be arbitrarily close to 1. To see this, take $w$ so that $w(1\otimes 1)\neq 0$.
One take direct sums of $w$'s and multiply $w$ on the right by elements of $L(H)$. 

Set $H^g := H \cap gHg^{-1}$. Since we assumed $B \not \prec L(H^g)$, take a sequence of unitaries $v_i \in B$ such that 
\[ \big\|E_{L(H^g)}(a^*v_ib) \big\|_2 \rightarrow 0 \: \text{for all} \: a,b \in M. \]
For $A \subset M$, use $E_A$ to still mean the conditional expectation $\text{id}\otimes E_A$ of $M_n(\mathbb{C}) \otimes M$ onto $M_n(\mathbb{C}) \otimes A$. We use this same convention for other maps that are extended to matrix spaces. 

\textbf{Claim:} For all $x \in M$, we have 
\[ \Big\|(1\otimes 1)E_{L(H)}\Big((1 \otimes u_g) w^*v_ix) \Big)\Big\|_2 \rightarrow 0.\]

For any set $L \subset G$, we denote by $p_L$ the orthogonal projection of $L^2(M)$ onto the closed linear span of $\{u_g | g \in L \}$. Let $t \in G$ be arbitrary and choose $\epsilon >0$. It is enough to show that 
\begin{equation}\label{eq:intersec}
    \Big\| (1 \otimes 1)E_{L(H)}\Big( (1 \otimes u_g)w^*v_iu_t \Big) \Big\|_2 \leq 3\epsilon 
\end{equation}
 for $i$ sufficiently large. Take a finite subset $K \subset G$ and $w \in M_{1,m}(\mathbb{C}) \otimes \text{Im} \: p_K$ such that $\|w - w_0 \|_2 < \epsilon$. Since $ E_{L(H)}\Big((1 \otimes u_g)w^*v_iu_t\Big) = p_H\Big((1 \otimes u_g)\theta(v_i)w^*u_t \Big) =  p_{H \cap gHKt} \Big( (1 \otimes u_g) \theta (v_i)w^*u_t \Big) $, it follows that 

 \[\left\| (1\otimes 1)E_{L(H)}\Big((1 \otimes u_g)w^*v_iu_t\Big) - (1\otimes 1) p_{H \cap gHKt} \Big( (1 \otimes u_g) \theta (v_i)w^*_0u_t \Big) \right\|_2 < \epsilon\]

 and \[ \left\| (1\otimes 1) p_{H \cap gHKt} \Big( (1 \otimes u_g) \theta (v_i)w^*_0u_t \Big)  - (1\otimes 1) p_{H \cap gHKt} \Big((1 \otimes u_g) w^*v_iu_t \Big) \right\|_2 < \epsilon. \]

Now take a finite set $L \subset G$ such that $H \cap gHKt = \sqcup _{s \in L} H^gs$. Then we get that 

\[ (1\otimes 1)p_{H \cap gHKt} \Big((1 \otimes u_g) w^*v_iu_t \Big) = \sum_{s \in L} (1 \otimes 1) E_{L(H^g)}\Big((1 \otimes u_g) w^*v_iu_{ts^{-1}} \Big)u_s.\]

The choice of $v_i$ implies that 
\[ \left\|\sum_{s \in L} (1 \otimes 1) E_{L(H^g)}\Big((1 \otimes u_g) w^*v_iu_{ts^{-1}} \Big)u_s\right\|_2 < \epsilon\] 
for $i$ sufficiently large. Combining this with the above estimates, we get that

\begin{align*}
    \| (1 \otimes 1)E_{L(H)}\Big( (1 \otimes u_g)w^*v_iu_t \Big) \|_2 & =  \left\|\sum_{s \in L} (1 \otimes 1) E_{L(H^g)}\Big((1 \otimes u_g) w^*v_iu_{ts^{-1}} \Big)u_s\right\|_2 + 2\epsilon \\
    & < 3 \epsilon
\end{align*}
holds for $i$ sufficiently large, which is exactly  \ref{eq:intersec} and proves the claim.

Now let $k \in \mathbb{N}$ and $a \in \mathcal{I}_k$. Take $\alpha: B \rightarrow M_k(\mathbb{C}) \otimes B$ such that $xa = a \alpha (x)$ for all $x \in B$. Note that 
\[ v_ia(1 \otimes w) = a(1\otimes w) (\text{id} \: \otimes \theta)\alpha(v_i).\]
Then using that $(\text{id} \: \otimes \theta)\alpha(v_i) \in M_{km}(\mathbb{C}) \otimes L(H)$ and $\alpha$ is an isometry, it follows that
\begin{align*}
    &\| (1 \otimes 1)E_{L(H)}\Big( (1 \otimes u_g)w^*v_ia(1 \otimes w) \Big) \|_2 \\
    & \quad = \| (1 \otimes 1)E_{L(H)}\Big( (1 \otimes u_g)w^*a(1\otimes w) (\text{id} \: \otimes \theta)\alpha(v_i)\Big)\|_2 \\
    & \quad = \| (1 \otimes 1)E_{L(H)}\Big( (1 \otimes u_g)w^*a(1 \otimes w) \Big) \|_2
\end{align*}
for all $i$. Applying the claim, we conclude that 
\[(1 \otimes 1)E_{L(H)}\Big( (1 \otimes u_g)w^*a(1 \otimes w) \Big) = 0 \quad \text{whenever} \: a \in I_k.\]
Since the matrix coefficients of $a$ for $a \in \mathcal{I}_k,$ $k \geq 1,$ span a strongly dense subset of $M$, it follows that $(1 \otimes 1)E_{L(H)}\Big( (1 \otimes u_g)w^*xw\Big) = 0$ for all $x \in M$. Let $q \in M$ be the smallest projection satisfying $w^*w \leq 1 \otimes q$. It follows that $E_{L(H)}(u_gqxq) = 0$ for all $x \in M$. Since $q$ can be taken arbitrarily close to 1, we have a contradiction. 
\end{proof}

For our main theorem we need a weaker version of malnormality. We say a subgroup $H$ of $G$ is \textit{weakly malnormal} if for some $n \in \mathbb{N}$ there exists $g_1,\dots ,g_n$ in $G \backslash H$ such that the subgroup $\cap_{i=1}^n H \cap g_iHg_i^{-1}$ is finite.

The following theorem is inspired by unpublished work of Chifan, Kunnawalkam Elayavalli, and the third author.

\begin{theorem}\label{thm:prime}
     Let $G$ be countable ICC group such that $L(G)$ does not have property Gamma. Suppose there exists $H\leq G$ such that $H$ is weakly malnormal and $L(H)$ is weakly bicentralized in $L(G)$. If $G$ admits an unbounded 1-cocycle into the quasi-regular representation $\ell^2(G/H)$, then $L(G)$ is prime.
\end{theorem}
We briefly remark that if
$H\leq G$ is a subgroup such that $L(H)$ is bicentralized in $L(G)$ (i.e.~ $[L(H)'\cap L(G)]' \cap L(G)=L(H)$), then it is also weakly bicentralized. Indeed, if $[L(H)'\cap L(G)]' \cap L(G)=L(H)$, then $[L(H)'\cap L(G)^\omega]' \cap L(G)=L(H)$ for any non-principal ultrafilter on $\mathbb{N}$.  In Section \ref{sec:PropsofV} we show that there are concrete subgroups of the Higman-Thompson groups $T_d$ and $V_d$ where the hypotheses of Theorem \ref{thm:prime} are realized. 
\begin{proof}
     Assume for the sake of contradiction that $M := L(G) = \mathcal P \bar\otimes \mathcal Q$ where both $\mathcal P$ and $\mathcal Q$ are II$_1$ factors. Then there are only two cases to consider. \\
         \textbf{Case I:} Suppose that $\mathcal P$ and $\mathcal Q$ are respectively not amenable relative to $L(H).$  Then by Lemma \ref{uniform-convergence-on-factors} we have that $\alpha_t \rightarrow Id$ uniformly both on $(\mathcal{Q})_1$ and on $(\mathcal{P})_1$, where $\alpha_t$ is the deformation described at the beginning of Section \ref{sec:defRigPrimenessQuasi}. Then by Lemma~\ref{lemma:convex}, $\alpha_t \rightarrow Id$ uniformly on $(M)_1$, but by Lemma~\ref{bounded-cocycle} this is a contradiction because the cocycle used to construct $\alpha_t$ is unbounded. 
\\
          \textbf{Case II:} Suppose without loss of generality that $\mathcal P$ is amenable relative to $L(H)$ in $M$. Since $M$ does not have property Gamma, using Connes' result \cite[Corollary 2.3]{connes1976classification} we get that  $\mathcal P$ is full. Now given that $\mathcal P$ is amenable relative $L(H)$ in $M$, $_ML^2(\mathcal M)_{\mathcal P} \prec {_ML^2 \langle M, e_{L(H)} \rangle_{\mathcal{P}}}$ as $M$-$\mathcal P$ bimodules, but this also means $_\mathcal{P}L^2(\mathcal P)_\mathcal{P} \prec {_\mathcal{P}L^2 \langle M, e_{L(H)} \rangle_\mathcal{P}}$ as $\mathcal P$-$\mathcal P$ bimodules. Together with the assumption that $L(H)$ is weakly bicentralized in $M$, Lemma~\ref{lemma:isono} now implies $ \mathcal P\prec_M L(H)$. Then, since $\mathcal P$ is regular in $M$, we can apply Lemma~\ref{lemma:intersec} to conclude that for all $g_1,\ldots, g_n \in G$, $P\prec_M L(\cap_{i=1}^n H \cap g_iHg_i^{-1})$. Since $H$ is weakly malnormal in $G$, there exists some $g_1,\ldots, g_n \in G \backslash H$ such that $|\cap_{i=1}^n H \cap g_iHg_i^{-1}| < \infty$. Since $\mathcal P$ is diffuse, this gives a contradiction. 
 \end{proof}

\section{Properties of the Higman-Thompson groups \texorpdfstring{$T_d$ and $V_d$}{Vd}}\label{sec:PropsofV}

The main goal of this section is to prove Theorem \ref{MainThm:Prime} by establishing that the Higman-Thompson groups $T_d$ and $V_d$ satisfy the conditions of Theorem \ref{thm:prime}. Section \ref{sec:Cocycles} is devoted to constructing the relevant cocycle. In Section \ref{sec:Primeness} Theorem \ref{MainThmPropP} and \ref{MainThm:Prime2} are established as corollaries to Lemma \ref{lem:notcoamen}, which proves the corresponding quasi-regular representation is non-amenable.

\subsection{Cocycle Construction}\label{sec:Cocycles}

In \cite{farley2003proper} Farley showed that $V=V_2$ admits a proper cocycle for a quasi-regular representation. Here we extend Farley's construction to $V_d,$ where $d \geq 2$ is an integer. The cocycle for $V_d$ is defined the same as the original, however, there are a few key adjustments needed for the proof to go through. For completeness and clarity, we include all of the lemmas and propositions of Farley's original paper, omitting proofs where no changes are necessary when moving from the dyadic to $d$-adic case. Hughes had previously shown $V_d$ admits a proper cocycle in \cite{hughes2009}.

\begin{definition}
    Let us start with several definitions which are the natural $d$-adic generalizations of the $d = 2$ case. 
    \begin{enumerate}
        \item A \textit{standard $d$-adic interval} is one of the from $[a/d^k, (a+1)/d^k],$ where $k$ and $a$ are integers such that $k \geq 0$ and $0 \leq a \leq d^k-1.$  Note that $d$-adic intervals can also be open or half open.  
        \item A \textit{standard $d$-adic partition} is a finite subset $\{x_0, \dots , x_n \}$ of $[0,1]$ such that $0 = x_0 < x_1 \dots < x_n =1,$ and $[x_i, x_{i+1}]$ is a standard $d$-adic interval for $0 \leq i \leq n-1.$  
        \item We call a standard $d$-adic partition an \textit{interval of a partition} if each $[x_i, x_{i+1})$ is half open for $0 \leq i \leq n-1.$ 
        \item Let $I$ be a standard $d$-adic interval. Then a \textit{standard affine map}, is a map $g: I \rightarrow \mathbb{R},$ such that $g$ is affine (i.e., $g(x) = ax+b$ for appropriate constants $a$ and $b$) and the image of $g$ is a standard $d$-adic interval. 
        \end{enumerate}
\end{definition}

It is important to note that the definition of a standard affine map is the natural generalization of the definition Fraley gave in \cite{farley2021corrigendum}. Farley wrote this corrigendum to the original paper \cite{farley2003proper} after the error and suggested corrections were pointed out by Nicolas Monod and Nansen Petrosyan. As the corrigendum states, the original proper cocycle proof for $d=2$ is not true unless the use of an affine map is replaced with \emph{standard affine map}. To that end, we emphasize the necessity of using standard affine maps in our following proof of the $d \geq 2$ case.

We can now define what we will prove is the proper cocycle for $V_d$. Define $V_{[0, 1/d)}$ to be the subgroup of $V_d$ whose elements act as the identity on $[0,1/d).$ Then define a subset 
$$X = \{gV_{[0, 1/d)} : g \: \text{is standard affine on [0, 1/d)} \}$$
of the homogeneous space $V_d / V_{[0, 1/d)}$. We can now define a map that we will justify is a proper cocycle. Let $c: V_d \rightarrow \ell^{\infty}(V_d / V_{[0, 1/d)})$ be defined by $c(v) = (v-1)\chi_X,$ where $\chi_X$ is the characteristic function of $X.$ As one can check, $c$ satisfies the cocycle condition: 
\begin{align*}v_1c(v_2)+c(v_1) &= v_1(v_2-1)\chi_X + (v_1-1)\chi_X \\ &= (v_1v_2-1)\chi_X \\
&= c(v_1v_2). 
\end{align*}
To argue $c$ is a proper cocycle for the quais-regular representation $\ell^2(V_d/V_{[0,1/d)})$ we need to show, (i) $c(v)$ is in $\ell^2(V_d/V_{[0,1/d)})$ for all $v \in V_d$ and (ii) $c: V_d \rightarrow \ell^2(V_d/V_{[0,1/d)})$ is a proper map.

\begin{proposition}(cf. \cite[Lemma 2.2]{cannon1996introductory})
Let $g \in V_d.$  There exists a standard $d$-adic partition $0 = x_0 < x_1 \dots < x_n =1,$ such that $g$ is standard affine on each interval of this partition and $\{g(x_0), \dots g(x_n)\}$ is a standard $d$-adic partition. 

\begin{proof}

Replace 2 with $d \in \mathbb{N}$ in the proof of  \cite[Proposition 2.1]{farley2003proper}.
\end{proof}

\end{proposition}

Now let $\mathcal{P}_1$ and $\mathcal{P}_2$ be two standard $d$-adic partitions of the unit interval and let $\phi$ be a bijection between the set of intervals of $\mathcal{P}_1$ and $\mathcal{P}_2.$  Then we call $T = (\mathcal{P}_1, \mathcal{P}_2, \phi)$ a \textit{representative triplet}.  The triplet $T$ determines an element $g_T$ of $V_d.$  We see that $g_T$ maps each interval $I$ of the partition $\mathcal{P}_1$ onto the interval $\phi(I)$ be an affine homeomorphism.  By the previous proposition, every element of $V_d$ is determined by a representative triplet.

We remark that part (3) of the next lemma is the first significant departure from Farley's $d=2$ proof.

\begin{lemma}\label{dpoint}(cf. \cite{cannon1996introductory}, remarks before Lemma 2.2, on page 220) Let $S$ be a standard $d$-adic partition. 
\begin{enumerate}
    \item If $(a/d^n, (a+1)/d^n) \cap S \neq \emptyset,$ where $(a/d^n, (a+1)/d^n)$ is a standard $d$-adic interval, then $a/d^n, (a+1)/d^n \in S$  
    \item If $b/d^k \in S$ and $b$ is not divisible by d, then $S\cap (0,1)$ contains at least $k$ elements. 
    \item If the standard $d$-adic interval $(c/d^l, (c+1)/d^l)$ contains an element of $S,$ then $(dc+n)/d^{l+1} \in S$ for at least one $n$ such that $1 \leq n < d.$  We say any such $(dc+n)/d^{l+1}$ is a $d^{th}$ point for $(c/d^l, (c+1)/d^l).$
\end{enumerate}

\begin{proof}
    (1) Replace $2$ with $d \in \mathbb{N}$ and ``odd'' with ``not divisible'' in the proof  of \cite[Lemma 2.2 part (1)]{farley2003proper}. %
  (2) Replace $2$ with $d \in \mathbb{N}$ and ``odd'' with ``not divisible'' in the proof  of \cite[Lemma 2.2 part (2)]{farley2003proper}. 

  (3) Choose some $x_j \in S \cap (c/d^l, (c+1)/d^l).$  If $x_j \neq (dc+1)/d^{l+1},$ then $x_j \in (c/d^l, (dc+1)/d^{l+1})$ or $((dc+1)/d^{l+1}, (c+1)/d^l)$.  If $x_j$ is in the first interval, which is a standard $d$-adic interval, then we can conclude $(dc+1)/d^{l+1} \in S$ by part (1). If $x_j$ is in the second interval, then $x_j$ is either contained in one of the standard $d$-adic intervals, $((dc+1)/d^{l+1}, (dc+2)/d^l), ((dc+2)/d^{l+1}, (dc+3)/d^l), \dots ((dc+(d-1))/d^{l+1}, (c+1)/d^l)$ or is one of the endpoints of these intervals.  If $x_j$ is one of the endpoints we are done.  If not, then by another application of part (1), we can conclude $((dc+n)/d^{l+1}) \in S$ for some $1 \leq n < d.$
\end{proof}
    
\end{lemma}

\begin{proposition}\label{prop}(cf. \cite{cannon1996introductory}, page 221).  Any $g \in V_d$ is determined by a unique representative triplet $T = (\mathcal{P}_1, \mathcal{P}_2, \phi)$ having the following property:
\begin{itemize}
    \item[*] If $I = [a/d^k, (a+1)/d^k)$ is any standard $d$-adic interval, then the restriction of $g$ to $I$ is standard affine if and only if the interior of $I$ contains no member of $\mathcal{P}_1.$\label{item:propertyStar}
\end{itemize}

\begin{proof}
    Replace $2$ with $d \in \mathbb{N}$ and ``affine'' with ``standard affine'' in the proof  of \cite[Proposition 2.3]{farley2003proper}.  
\end{proof}
\end{proposition}

The following theorem proves that $c$ is a proper cocycle for the quasi-regular representation.  The key difference with Farley's original argument is the functions $f_i$ defined in equation \ref{fi}.  Essentially, we had to recover uniqueness that one loses when moving from a midpoint to general $d$-th points. The rest of the argument is the same but we include it for completeness. Note that $V_d$ and its subgroups were first shown to have the Haagerup property in \cite{hughes2009}. 

\begin{theorem}\label{cocycle}
    For any $v \in V_d,$ we have that (i) $c(v) \in \ell^2(V_d/V_{[0,1/d)}),$ and (ii) the function $c:V \rightarrow \ell^2(V_d/V_{[0,1/d)})$ is a proper cocycle for the quasi-regular representation of $V_d$ on $\ell^2(V_d/V_{[0,1/d)}).$  In particular, $V_d$ (and therefore each of $F_d$ and $T_d$) has the Haagerup property. 

\begin{proof}
Let $v \in V_d$ and let $T = (\mathcal{P}_1, \mathcal{P}_2, \phi)$ be the representative triplet for $v$ satisfying property (*) from Proposition \ref{prop}.  For $i =1,2,$ let 
\[X_{\mathcal{P}_i} = \{gV_{[0,1/d)}: g \: \text{maps} \:[0,1/d) \: \text{standard affinely and} \: g((0,1/d)) \cap \mathcal{P}_i = \emptyset \}.\]
(i) To prove the first assertion, we will use the subspaces $X_{\mathcal{P}_i}$, which is dependent on the element $v \in V_d$, to rewrite the characteristic function defining $c$ in such a way that we can argue $c$ is finitely supported.  

 We first claim the bijection $v: V_d/V_{[0,1/d)} \rightarrow V_d/V_{[0,1/d)}$ induced by left multiplication by $v,$ provides a bijection from $X_{\mathcal{P}_1}$ onto $X_{\mathcal{P}_2}.$  Indeed, suppose $gV_{[0,1/d)} \in X_{\mathcal{P}_1},$ then by definition, $g$ maps $[0, 1/d)$ standard affinely and $g((0,1/d)) \cap \mathcal{P}_1 = \emptyset.$  Next, property (*) implies that $vg$ also maps $[0, 1/d)$ standard affinely and $vg((0,1/d)) \cap \mathcal{P}_2 = \emptyset.$ Which by definition, means $vgV_{[0,1/d)} \in X_{\mathcal{P}_2}.$  A similar argument tells us, if $gV_{[0,1/d)} \in X_{\mathcal{P}_2},$ then $v^{-1}gV_{[0,1/d)}$ is in $X_{\mathcal{P}_1}.$ Thus, $v \cdot X_{\mathcal{P}_1} \mapsto X_{\mathcal{P}_2}$ is a bijection.  This implies that $v \cdot \chi_{X_{\mathcal{P}_1}} - \chi_{X_{\mathcal{P}_2}} = \chi_{v \cdot (X_{\mathcal{P}_1})} - \chi_{X_{\mathcal{P}_2}} = 0.$ 

We use this decomposition to rewrite $c(v):$
 \begin{align}\label{co1}
c(v) = (v-1)\chi_X &= \chi_{v\cdot X}- \chi_X \\ 
&= \chi_{v\cdot X}- \chi_X + \chi_{X_{\mathcal{P}_2}} - \chi_{v \cdot (X_{\mathcal{P}_1})} \\
&=   \chi_{v \cdot (X-X_{\mathcal{P}_1})} - \chi_{X-X_{\mathcal{P}_2}} \label{co}
 \end{align}

Now note an element of $gV_{[0,1/d)}$ of $X$ is determined by the interval $g([0,1/d)),$ because every element of  $V_{[0,1/d)}$ fixes $[0,1/d).$  With that in mind, for $i=1,2$ we define functions 

\begin{equation}\label{fi}
    f_i:X-X_{\mathcal{P}_i} \rightarrow \mathcal{P}_i-\{0, 1/d, 2/d, \dots, (d-1)/d, 1\}
\end{equation}
by sending an element $gV_{[0,1/d)}$ of $X-X_{\mathcal{P}_i}$ to 

\[ \min \Bigg\{ \Big\{ \frac{n}{d}(g(0)+g(1/d)) : 1 \leq n \leq d\Big\} \cap \mathcal{P}_i\Bigg\}. \]

In other words, $f_i$ is sending $gV_{[0,1/d)}$ to the smallest $d^{th}$ point of the standard $d$-adic interval, $g([0,1/d))$. To see $f_i$ is well defined, note how $g([0,1/d))$ contains a member of $\mathcal{P}_i$ by the definition of $X-X_{\mathcal{P}_i},$ and therefore, by Lemma \ref{dpoint} (3), we know $\{ (n/d)(g(0)+g(1/d)) : 1 \leq n \leq d\} \cap \mathcal{P}_i \neq \emptyset$ because $(n/d)(g(0)+g(1/d)) $ corresponded to a $d^{th}$ points of $g([0,1/d))$. Taking the minimum over this intersection now guarantees $f_i$ can only take one value in $\mathcal{P}_i.$  Now $f_i(gV_{[0,1/d)})$ is not $0$ or $1,$ and moreover, it is not in $\{0, 1/d, 2/d, \dots, (d-1)/d, 1\}$, because $g$ is bijective.  Indeed, if $f_i(gV_{[0,1/d)}) = (n/d)$ for any $1 \leq n < d,$ then this implies $g([0,1/d)) = [0,1)$ or $[0,1].$  In either case, $g([1/d,1]) = 1$ or the empty set, contradicting the bijectivity of $g.$ We can now conclude $f_i$ is well defined.

Next, we want to show each $f_i$ is a bijection.  To see injectivity, we remark that a standard $d$-adic interval is determined by knowing any of its $d^{th}$ points.  To see this, recall how the endpoints of a standard $d$-adic interval, $[a/d^l, (a+1)/d^l]$ are reduced fractions, so if you know $(n/d)(g(0)+g(1/d))$ for some $1 \leq n < d,$ you can recover all of $g([0,1/d)).$  To see that $f_i$ is surjective, note that every $d$-adic rational other than, $0, 1/d, 2/d, \dots, (d-1)/d$ or $1$ is the $d^{th}$ point of some standard $d$-adic interval $I=[a/d^l, (a+1)/d^l).$ Additionally, any standard $d$-adic interval can be viewed as the image  $h([0,1/d)),$ for some $h$ which maps $[0,1/d)$ standard affinely. Moreover, such an $h$ is in $X-X_{\mathcal{P}_i}$ exactly because it contains an element of $\mathcal{P}_i.$

Next note how the image of $f_i ( X-X_{\mathcal{P}_i})$ is a finite set and by the decomposition of the cocycle in equation \ref{co}, $c(v)$ is non-zero only finitely many times.  Thus, $c(v)$ is finitely supported and hence an element of $\ell^2(V_d/V_{[0,1/d)}).$

(ii) We will now prove that $c$ is proper, by showing that there can only be finitely many standard $d$-adic partitions of a given size.  First, $\|c(v) \|_2$ is determined by $\mathcal{P}_1$ and $\mathcal{P}_2$ from it's representative triple.  To see this, take $gV_{[0,1/d)} \in X - X_{\mathcal{P}_1}.$  By definition, $g$ maps $[0,1/d)$ standard affinely, $g([0, 1/d)) \cap \mathcal{P}_1 \neq \emptyset,$ and $g([0, 1/d))$ is a standard $d$-adic interval. Therefore, $v$ cannot map $g([0, 1/d))$ standard affinely by property (*) from proposition \ref{prop}. It follows that $vgV_{[0,1/d)} \notin X$. Therefore, $v\cdot (X-X_{\mathcal{P}_1}) \cap (X-X_{\mathcal{P}_2}) = \emptyset.$ Again using the decomposition of the cocycle in equation \ref{co}, and the bijective nature of $f_i$ we can see that 
\begin{align}
    \|c(v) \|_2 &= \sqrt{|X-X_{\mathcal{P}_1}|+|X-X_{\mathcal{P}_2}|} \\
    &= \sqrt{|\mathcal{P}_1-\{0, 1/d, 2/d, \dots, 1\}| +|\mathcal{P}_2-\{0, 1/d, 2/d, \dots, 1\}|.}
\end{align}

Let $S$ be a standard $d$-adic partition containing $k+2$ elements.  Let $b/d^l \in S$ for some integer $l \leq k$ such that $b \centernot| d.$  Then by Lemma \ref{dpoint} (2), $S \cap (0,1)$ contains $l$ elements. 
Hence, $S$ only contains $d$-adic rational numbers of the form $a/d^l,$ where $l \leq k.$ Consequently, for any positive integer, there can only be finitely many standard $d$-adic partitions of $[0,1]$ of that size.  Therefore, for any $r \in \mathbb{R},$ there are only finitely many representative triplets $(\mathcal{P}_1, \mathcal{P}_2, \phi)$ such that  
\[ \sqrt{|\mathcal{P}_1-\{0, 1/d, 2/d, \dots, 1\}| +|\mathcal{P}_2-\{0, 1/d, 2/d, \dots, 1\}|} < r.\]
In conclusion, $c:V \rightarrow \ell^2(V_d/V_{[0,1/d)})$ is a proper cocycle.
\end{proof}

\end{theorem}

\subsection{Primeness and Proper Proximality}\label{sec:Primeness}

In this section, we collect the proofs of the remaining results for the Higman-Thompson groups $T_d$ and $V_d$.

\begin{lemma}\label{lem:notcoamen}
   The subgroup $V_{[0,1/d)}$ (respectively $T_{[0,1/d)]}$) is not co-amenable inside $V_d$ (respectively $T_d$). In particular, the corresponding quasi-regular representation $\ell^2 (V_d / V_{[0,1/d)})$ (respectively $\ell^2 (T_d / T_{[0,1/d)})$) is non-amenable. 
\end{lemma}

    \begin{proof}

We will show the action of $V_d$ on $V_d/V_{[0,1/d)}$ contains a paradoxical decomposition. 

Let $X_1 = [0, \frac{1}{d^2})$ and $X_2 = [\frac{1}{d^2}, \frac{2}{d^2})$. Consider the following two elements $h_1, h_2 \in V_d$, which should be understood as carrying out a ``nesting'' of $X_1$ and $X_2$. 
See Figure \ref{fig:noncoamean} for a depiction of $h_1$ and $h_2$ as maps between $d$-ary trees for $d=3$.

\begin{equation}\label{eq:noncoamean}
h_1 = \begin{cases}
    [0, \frac{1}{d^2}) \mapsto [\frac{d}{d^3}, \frac{d+1}{d^3} )\\
    [\frac{1}{d^2},\frac{2}{d^2}) \mapsto [\frac{d+1}{d^3}, \frac{d+2}{d^3} ) \\ 
    \vdots \\
    [\frac{d-1}{d^2}, \frac{d}{d^2} ) \mapsto [\frac{2d-1}{d^3}, \frac{2d}{d^3} ) \\
     [\frac{d}{d^2}, \frac{d+1}{d^2} ) \mapsto  [\frac{d-1}{d^2}, \frac{d}{d^2}) \\
     [\frac{d+1}{d^2}, \frac{d+2}{d^2} ) \mapsto  [\frac{1}{d}, \frac{2}{d})\\
     \vdots \\
     [\frac{2d-1}{d^2}, \frac{2d}{d^2} ) \mapsto  [\frac{2}{d}, \frac{3}{d}) \\
     [\frac{2}{d}, \frac{3}{d}) \mapsto [\frac{3}{d}, \frac{4}{d}) \\
     \vdots \\
      [\frac{d-3}{d}, \frac{d-2}{d}) \mapsto [\frac{d-2}{d}, \frac{d-1}{d}) \\
        [\frac{d-2}{d}, \frac{d-1}{d}) \mapsto [\frac{d-1}{d}, 1) \\
      [\frac{d-1}{d}, 1) \mapsto [0, \frac{1}{d^2})  
\end{cases} 
 h_2 = \begin{cases}
    [0, \frac{1}{d^2}) \mapsto [0, \frac{1}{d^3} )\\
    [\frac{1}{d^2},\frac{2}{d^2}) \mapsto [\frac{1}{d^3}, \frac{2}{d^3} ) \\ 
    \vdots \\
    [\frac{d-1}{d^2}, \frac{d}{d^2} ) \mapsto [\frac{d-1}{d^3}, \frac{d}{d^3} )  \\
    [\frac{d}{d^2}, \frac{d+1}{d^2} ) \mapsto [\frac{1}{d^2}, \frac{2}{d^2} ) \\
    \vdots \\
    [\frac{2d-2}{d^2}, \frac{2d-1}{d^2} ) \mapsto [\frac{d-1}{d^2}, \frac{d}{d^2} ) \\
     [\frac{2d-1}{d^2}, \frac{2d}{d^2} ) \mapsto [\frac{1}{d}, \frac{2}{d} ) \\
     [\frac{2}{d}, \frac{3}{d} ) \mapsto [\frac{2}{d}, \frac{3}{d} ) \\
     \vdots \\
     [\frac{d-1}{d}, 1 ] \mapsto [\frac{d-1}{d}, 1 ] 
\end{cases} 
\end{equation}

Note that $X_1 \cap X_2 = \emptyset$, $h_1(X_1) \subset X_2$, and $h_2(X_2) \subset X_1$. Now define the following subsets of $V_d/V_{[0,1/d)}$:
\[ A = \{gV_{[0,1/d]} : g([0, 1/d)) \subset X_1 \}, \quad  B = \{gV_{[0,1/d]} : g([0, 1/d)) \subset X_2 \}.\]
Then for $gV_{[0,1/d]} \in A$, we have that $h_1\circ g([0,1/d]) \subset X_2$. So $h_1(A) \subset B$. Similarly, for $gV_{[0,1/d]} \in B$, we have that $h_1\circ g([0,1/d]) \subset X_1$. So $h_2(B) \subset A$. It can also be seen that $h_1^n(A) \subset B$ and $h_2^n(B) \subset A$ for all $n \in \mathbb{N}$. Therefore by the ping-pong lemma, we have that $\langle h_1, h_2\rangle \cong \mathbb{F}_2$ as a subgroup of $V_d$. Now define several more subsets of $V_d/V_{[0,1/d)}$: 
\begin{align*}
    Y &= \{ wV_{[0,1/d]}: w \: \text{is a word in} \: \langle h_1, h_2 \rangle \} \\
    G_1 &= \{ h_1wV_{[0,1/d]}: w \: \text{is a word in} \: \langle h_1, h_2 \rangle \} \\
    G_2 &= \{ h_1^{-1}wV_{[0,1/d]}: w \: \text{is a word in} \: \langle h_1, h_2 \rangle \} \\ 
    G_3 &= \{ h_2wV_{[0,1/d]}: w \: \text{is a word in} \: \langle h_1, h_2 \rangle \} \\
    G_4 &= \{ h_2^{-1}wV_{[0,1/d]}: w \: \text{is a word in} \: \langle h_1, h_2 \rangle \}.  
\end{align*}

Then $Y = G_1 \sqcup G_2 \sqcup G_3 \sqcup G_4$ is a paradoxical decomposition for $Y \subset V_d/V_{[0,1/d)}$. Therefore we concluded that $V_{[0,1/d)}$ is not co-amenable inside $V_d$.

Now note that $h_1$ and $h_2$ are cyclic bijections of $[0,1]$ so they are also elements of $T_d$. By replacing $V_{[0,1/d)}$ with $T_{[0,1/d)}$ in the above argument we quickly see that $T_{[0,1/d)}$ is not co-amenable in $T_d$.
\end{proof}

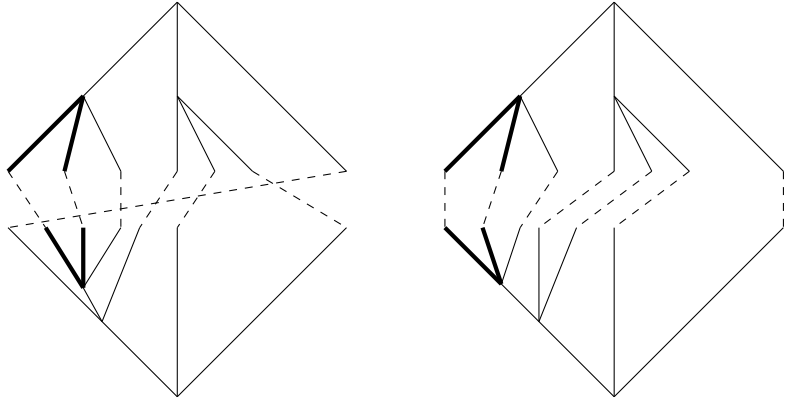
\begin{figure}
    \centering
    \begin{tikzpicture}
    \draw (5.75,10.5) -- (4.5,9.25);
    \draw[ultra thick] (4.5,9.25) -- (3.5,8.25);
    \draw  (5.75,10.5) -- (8,8.25);
    \draw  (5.75,10.5) -- (5.75,8.25);
    \draw  (4.5,9.25) -- (5,8.25);
    \draw[ultra thick]  (4.5,9.25) -- (4.25,8.25);
    \draw  (5.75,9.25) -- (6.75,8.25);
    \draw  (5.75,9.25) -- (6.25,8.25);
    \draw  (3.5,7.5) -- (5.75,5.25);
    \draw  (8,7.5) -- (5.75,5.25);
    \draw  (5.75,7.5) -- (5.75,5.25);
    \draw (4.75,6.25) -- (5.25,7.5);
    \draw  (4.75,6.25) -- (4.5,6.7);
    \draw[ultra thick] (4.5,6.7) -- (4,7.5);
    \draw   (4.5,6.7) -- (5,7.5);
    \draw [ultra thick] (4.5,6.7) -- (4.5,7.5);
    \draw [dashed] (3.5,8.25) -- (4,7.5);
    \draw [dashed] (4.25,8.25) -- (4.5,7.5);
    \draw [dashed] (5,8.25) -- (5,7.5);
    \draw [dashed] (5.75,8.25) -- (5.25,7.5);
    \draw [dashed] (6.25,8.25) -- (5.75,7.5);
    \draw [dashed] (6.75,8.25) -- (8,7.5);
    \draw [dashed] (8,8.25) -- (3.5,7.5);
    \end{tikzpicture}
\hspace{1cm}
\begin{tikzpicture}
    \draw  (5.75,10.5) -- (3.5,8.25);
    \draw[ultra thick] (4.5,9.25) -- (3.5,8.25);
\draw  (5.75,10.5) -- (8,8.25);
\draw  (5.75,10.5) -- (5.75,8.25);
\draw  (4.5,9.25) -- (5,8.25);
 \draw[ultra thick]  (4.5,9.25) -- (4.25,8.25);
\draw  (4.5,9.25) -- (4.25,8.25);
\draw  (5.75,9.25) -- (6.75,8.25);
\draw  (5.75,9.25) -- (6.25,8.25);
\draw  (3.5,7.5) -- (5.75,5.25);
\draw [ultra thick] (3.5,7.5) -- (4.25,6.75);
\draw  (8,7.5) -- (5.75,5.25);
\draw  (5.75,7.5) -- (5.75,5.25);
\draw  (4.75,6.25) -- (5.25,7.5);
\draw  (4.75,6.25) -- (4.75,7.5);
\draw  (4.25,6.75) -- (4.5,7.5);
\draw  [ultra thick] (4,7.5) -- (4.25,6.75);
\draw [dashed] (3.5,8.25) -- (3.5,7.5);
\draw [dashed] (4.25,8.25) -- (4,7.5);
\draw [dashed] (5,8.25) -- (4.5,7.5);
\draw [dashed] (5.75,8.25) -- (4.75,7.5);
\draw [dashed] (6.25,8.25) -- (5.25,7.5);
\draw [dashed] (6.75,8.25) -- (5.75,7.5);
\draw [dashed] (8,8.25) -- (8,7.5);
\end{tikzpicture}
    \caption{Picture above are $h_1$ and $h_2$ from equation \ref{eq:noncoamean} for $d=3$. On the left is $h_1$, mapping both $X_1$ and $X_2$ into $X_2$. On the right is $h_2$, mapping both $X_1$ and $X_2$ into $X_1$.}
    \label{fig:noncoamean}
\end{figure}

Any group that admits a proper cocycle into a non-amenable representation has, by definition, property (HH) of Ozawa and Popa \cite{ozawapopa2010cartanII}. Moreover, by \cite[Proposition 2.1]{ozawapopa2010cartanII}, any group with property (HH) is not inner amenable. So combining this with Lemma \ref{lem:notcoamen} and Theorem \ref{cocycle} gives a proof of the non-inner amenability of $T_d$ and $V_d$ which is different than the proofs given in \cite{HO2016} and \cite{bashwinger2022non}. Boutonnet, Ioana, and Peterson showed in \cite[Proposition 1.6]{boutonnet2021properly} that groups with property (HH) are also properly proximal, so this same discussion gives the following corollary.

\begin{corollary}[Theorem \ref{MainThmPropP}]\label{Cor:PropP}  The Higman-Thompson group $T_d$ and $V_d$ are properly proximal for every integer $d\geq 2$.
\end{corollary}

It is interesting to observe the proper proximality of $T_d$ and $V_d$, as neither group is bi-exact or acylindrically hyperbolic (see \cite{dahmani2017hyperbolically}).

If a group $G$ admits an unbounded cocycle into a non-amenable orthogonal representation that is mixing with respect to some family of subgroups $\mathcal{S}$, then $G$ satisfies a special case of the \textbf{NC}$(\mathcal{S})$ condition of \cite{chifan2016inner}. Therefore, combining Theorem \ref{cocycle}, Lemma \ref{lem:notcoamen}, and Lemma \ref{lem:relmix} with \cite[Theorem 4.1, Corollary 4.2]{chifan2016inner} yields the following corollary.

\begin{corollary}\label{cor:VnonGamma}
Let $(N,\tau)$ be any tracial von Neumann algebra, and let $V_d \curvearrowright (N, \tau)$ (respectively,  $T_d \curvearrowright (N, \tau)$)
be any free strongly ergodic p.m.p~action. Then $N \rtimes V_d$ (respectively, $N \rtimes T_d$) does not have property Gamma.

\end{corollary}

Recall that $N \subset M$ is {weakly bicentralized} in $M$ if $_ML^2\langle M, e_N \rangle_M \prec {_ML^2(M)_M}$.
To prove $L(V_d)$ is prime by applying Theorem \ref{thm:prime}, will show that $L(V_{[0, 1/d)})$ is weakly bicentralized in $L(V_d)$, in fact we will show it is bicentralized, which is a prior stronger.   
By \cite[Lemma 3.4]{BannonFull2020}, it suffices to show that  \[L(V_{[0,1/d]})=(L(V_{[0,1/d]})'\cap L(V_{d}))'\cap L(V_{d}).\]

To that end, we will need to have some control over $L(H)' \cap L(V_d)$ for specific subgroups $H\leq V_d$. Letting $G$ be a group and $H$ a subgroup of $G$, recall that that the centralizer of $H$ in $G$ is \[\mathcal{C}_G(H):=\{g\in G: h^{-1}gh=g\, \forall h\in H\}.\]
The virtual centralizer  $H$ in $G$ is defined as
\[\mathcal{VC}_G(H) := \{g\in G : |\{h^{-1}gh: h \in H\}| < \infty \}.\]
It is an exercise to show that
\[L(H)'\cap L(G)= \left\{ \sum_{g\in G }c_g \lambda_g\in L(G) :  g\in \mathcal{VC}_G(H),\,  c_{g}=c_{h^{-1}gh}\,\forall g\in G,h\in H\right\}.\]

Recall the subgroup $V_{[0,1/d)}$ defined in Section \ref{sec:Cocycles} is the subgroup of elements of $V_d$ that act as the identity on the interval $[0,1/d)$. Similarly, an element of the $V_{[1/d, 1]}$ acts as the identity on $[1/d,1].$

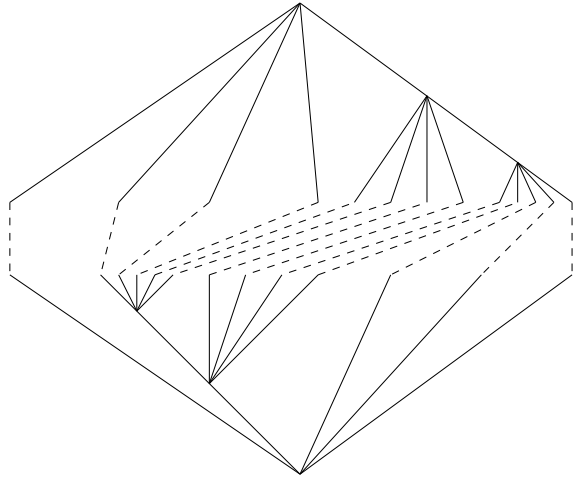
\begin{figure}
\centering
\resizebox{0.5\textwidth}{!}{%
\begin{tikzpicture}
\draw  (2,8.5) -- (6,11.25);
\draw  (6,11.25) -- (9.75,8.5);
\draw  (3.5,8.5) -- (6,11.25);
\draw  (4.75,8.5) -- (6,11.25);
\draw  (6.25,8.5) -- (6,11.25);
\draw  (9,9.05) -- (9.5,8.5);
\draw  (9,9.05) -- (9.25,8.5);
\draw  (9,8.5) -- (9,9.05);
\draw  (8.75,8.5) -- (9,9.05);
\draw  (7.75,9.97) -- (6.75,8.5);
\draw  (7.75,9.97) -- (7.25,8.5);
\draw  (7.75,9.97) -- (7.75,8.5);
\draw  (8.25,8.5) -- (7.75,9.97);
\draw  (2,7.5) -- (6,4.75);
\draw  (9.75,7.5) -- (6,4.75);
\draw  (6,4.75) -- (3.25,7.5);
\draw  (3.75,7) -- (3.5,7.5);
\draw  (3.75,7) -- (3.75,7.5);
\draw  (3.75,7) -- (4,7.5);
\draw  (3.75,7) -- (4.25,7.5);
\draw  (4.75,7.5) -- (4.75,6);
\draw  (4.75,6) -- (5.25,7.5);
\draw  (4.75,6) -- (5.75,7.5);
\draw  (4.75,6) -- (6.25,7.5);
\draw  (6,4.75) -- (7.25,7.5);
\draw  (6,4.75) -- (8.5,7.5);
\draw [dashed] (2,8.5) -- (2,7.5);
\draw [dashed] (3.5,8.5) -- (3.25,7.5);
\draw [dashed] (4.75,8.5) -- (3.5,7.5);
\draw [dashed] (6.25,8.5) -- (3.75,7.5);
\draw [dashed] (4,7.5) -- (6.75,8.5);
\draw [dashed] (7.25,8.5) -- (4.25,7.5);
\draw [dashed] (4.75,7.5) -- (7.75,8.5);
\draw [dashed] (8.25,8.5) -- (5.25,7.5);
\draw [dashed] (8.75,8.5) -- (5.75,7.5);
\draw [dashed] (6.25,7.5) -- (9,8.5);
\draw [dashed] (9.25,8.5) -- (7.25,7.5);
\draw [dashed] (9.5,8.5) -- (8.5,7.5);
\draw [dashed] (9.75,8.5) -- (9.75,7.5);
\end{tikzpicture}
}%
\caption{Pictured above is $v_2$ from the sequence $(v_n) \in V_{[0,1/d)} \cap F_d$ for $d=5$ from Remark \ref{rmk:BicentSeq}.}
\label{fig:bicent}
\end{figure}

\begin{lemma}\label{lemma:VC} For $V_d$, with subgroups $V_{[0, 1/d)}, V_{[1/d, 1]} < V_d$ and $T_d$, with subgroups $T_{[0, 1/d)}, T_{[1/d, 1]} < T_d$  the following hold: 
    \begin{enumerate}
        \item $\mathcal{VC}_{V_d}(V_{[0, 1/d)})=\mathcal{C}_{V_d}(V_{[0, 1/d)}) = V_{[1/d,1]}$, \label{item:VCV_0}
        \item $\mathcal{VC}_{V_d}(V_{[1/d,1]})=\mathcal{C}_{V_d}(V_{[1/d,1]})= V_{[0,1/d)}$, \label{item:VCV_1}
         \item $\mathcal{VC}_{T_d}(T_{[0, 1/d)})=\mathcal{C}_{T_d}(T_{[0, 1/d)}) = T_{[1/d,1]}$, \label{item:VCT_0}
        \item $\mathcal{VC}_{T_d}(T_{[1/d,1]})=\mathcal{C}_{T_d}(T_{[1/d,1]})= T_{[0,1/d)}$.\label{item:VCT_1}
    \end{enumerate}
    In particular, $(L(V_{[0, 1/d)})'\cap L(V_d))'\cap L(V_d)=L(V_{[0, 1/d)})$ and $(L(T_{[0, 1/d)})'\cap L(T_d))'\cap L(T_d)=L(T_{[0, 1/d)})$.
\end{lemma}

\begin{proof}
For brevity, we denote $g^H=\{h^{-1}gh: h \in H\}$.
        To prove \ref{item:VCV_0}, first note the centralizer of $V_{[0, 1/d)}$ is $V_{[1/d, 1]}$. This is clear given that for $h \in V_d$, $h^{-1}gh = g$ for all $g \in V_{[0, 1/d)}$ if and only if $h$ is the identity on $[1/d,1]$, meaning $h$ must be in $V_{[1/d, 1]}$. Next, we show the virtual centralizer is also equal to $V_{[1/d, 1]}$. To show $\mathcal{C}_{V_d}(V_{[0, 1/d)}) = \mathcal{VC}_{V_d}(V_{[0, 1/d)})$, it is enough to show for every $g \in V_d \backslash \mathcal{C}_{V_d}(V_{[0,1/d]})$, $|g^{V_{[0, 1/d)}}| = \infty$; equivalently, if $g \in V_d \backslash \mathcal{C}_{V_d}(V_{[0, 1/d)}),$ there exists $(v_n)_{n\in \mathbb{N} }\in V_{[0, 1/d)}$ such that $v_n^{-1}gv_n \neq v_m^{-1}gv_m$ for infinitely many $n \neq m.$ \\

    Case 1: If $g \not \in \mathcal{C}_{V_d}(V_{[0, 1/d)})$ but $g \in V_{[0, 1/d)},$ then $|g^{V_{[0, 1/d)}}| = \infty$ since $V_{[0, 1/d)}$ is an ICC subgroup of $V_d$. \\

    Case 2: If $g \not \in \mathcal{C}_{V_d}(V_{[0, 1/d)})$ and $g \not \in V_{[0, 1/d)},$ then we can take the sequence $(v_n)_{n \in \mathbb N} \in V_{[0, 1/d)} \cap F_d$ such that $v_n$ is not fixing any $d$-ary interval inside $[1/d,1]$ and $v_n \neq v_m$ for all $n \neq m$. See Remark \ref{rmk:BicentSeq} below for a explicit construction of such a sequence. Now by assumption, $g$ cannot be the identity on $[0,1/d)$ or $[1/d,1]$ and $v_n^{-1}gv_n = v_m^{-1}gv_m$ is only possible when $g$ is the identity on $[1/d,1]$. 
    Therefore $v_n^{-1}gv_n \neq v_m^{-1}gv_m$ for infinitely many $n \neq m$, as needed.

    This completes the proof of \eqref{item:VCV_0}. The proof of \eqref{item:VCV_1} then follows from a symmetric argument. Note that \eqref{item:VCV_0} implies that $L(V_{[0,1/d)})'\cap L(V_d)= L(V_{[1/d,1]})$, and \eqref{item:VCV_1} implies that $ L(V_{[1/d,1]})'\cap L(V_d)= L(V_{[0,1/d)}) $. Thus $(L(V_{[0, 1/d)})'\cap L(V_d))'\cap L(V_d)=L(V_{[0, 1/d)})$.

    Each element in the sequence $(v_n)_{n \in \mathbb{N}}$ defined above in case 2 is an element of $F_d$, therefore it is also in $T_d$. With that key observation, and replacing $V_{[0, 1/d)}, V_{[1/d, 1]}, V_d$ with $T_{[0,1/d)}, T_{[1/d,1]}, T_d$ where appropriate, \eqref{item:VCT_0} and \eqref{item:VCT_1} also follow. 
\end{proof}

\begin{remark}\label{rmk:BicentSeq}
     We can construct a sequence $(v_n)_{n \in \mathbb N} \in V_{[0, 1/d)} \cap F_d$ (so also in $T_d$) such that  $v_n$ is not fixing any $d$-ary interval inside $[1/d,1]$ and $v_n \neq v_m$ for all $n \neq m$. One such sequence can be described as follows.
     
     Viewing $V_d$ as the map of piecewise bijections from the unit interval to itself, define $v_1$ by dividing $(\frac{d-1}{d}, 1)$ into $d$ equal parts in the domain of $v_1$ and $(\frac{1}{d}, \frac{2}{d})$ into $d$ equal parts in the range. Then $v_1$ will be the element of $F_d$ that is the identity on $[0,1/d)$ and connects the remaining $d$-ary intervals as defined above in a piecewise continuous way. Define $v_2$ recursively from $v_1$ by dividing $(\frac{d^2-1}{d^2}, 1)$ into $d$ equal parts in the domain of $v_2$ and $(\frac{1}{d^2}, \frac{2}{d^2})$ into $d$ equal parts in the range. Following this procedure, $v_n$ is defined by dividing $(\frac{d^n-1}{d^n}, 1)$ into $d$ equal parts in the domain of $v_n$ and $(\frac{1}{d^n}, \frac{2}{d^n})$ into $d$ equal parts in the range and connecting as before. It is straightforward to see that each $v_n$ will not fix any $d$-ary interval inside $[1/d,1]$ and $v_n \neq v_m$ for all $n \neq m$. 

     For clarity, the equation for $v_2$ is described in Equation \eqref{equation:v2swq} 
     where there are six distinct piecewise patterns. In general, $v_n$ contains $n+4$ distinct piecewise patterns. The intuition behind the sequence can be seen in Figure \ref{fig:bicent} where $v_2$ is depicted for $d=5$. 
\begin{figure*}
    \centering
    \begin{equation} \label{equation:v2swq}
    v_2 = \begin{cases}
    [0,1/d) \mapsto [0,1/d) \\
    [\frac{1}{d}, \frac{2}{d}) \mapsto [\frac{d^2}{d^3}, \frac{d^2+1}{d^3}) \\
    \vdots \\
     [\frac{d-2}{d},\frac{d-1}{d}) \mapsto [\frac{d^2+d-3}{d^3},\frac{d^2+d-2}{d^3}) \\
     [\frac{d^2-d}{d^2}, \frac{d^2-d+1}{d^2}) \mapsto [\frac{d^2+d-2}{d^3},\frac{d^2+d-1}{d^3}) \\
    [\frac{d^2-d+1}{d^2}, \frac{d^2-d+2}{d^2}) \mapsto [\frac{d^2+d-1}{d^3},\frac{d^2+d}{d^3}) \\
    [\frac{d^2-d+2}{d^2}, \frac{d^2-d+3}{d^2}) \mapsto [\frac{d+1}{d^2},\frac{d+2}{d^2}) \\
    \vdots \\
     [\frac{d^2-2}{d^2}, \frac{d^2-1}{d^2}) \mapsto [\frac{2d-3}{d^2},\frac{2d-2}{d^2}) \\
    [\frac{d^3-d}{d^3}, \frac{d^3-d+1}{d^3}) \mapsto [\frac{2d-2}{d^2}, \frac{2d-1}{d^2}) \\
     [\frac{d^3-d+1}{d^3}, \frac{d^3-d+2}{d^3}) \mapsto [\frac{2d-1}{d^2}, \frac{2d}{d^2}) \\
    [\frac{d^3-d+2}{d^3}, \frac{d^3-d+3}{d^3}) \mapsto [\frac{2}{d}, \frac{3}{d}) \\
    \vdots \\
    [\frac{d^3-1}{d^3}, 1] \mapsto [\frac{d-1}{d}, 1] \\
\end{cases}
\end{equation}
    \caption{The equation of $v_2$ from the sequence $(v_n) \in V_{[0,1/d)}\cap F_d$ from Remark \ref{rmk:BicentSeq}.}
    \label{fig:v2seq}
\end{figure*}

\end{remark}

We now have all of the necessary tools to prove $L(V_d)$ and $L(T_d)$ are prime. 

\begin{corollary}[Theorem \ref{MainThm:Prime}]\label{cor:Prime}
    The group von Neumann algebra $L(V_d)$ (respectively $L(T_d$)) of the Higman-Thompson group $V_d$ (respectively $T_d$) is prime for every integer $d\geq 2$. 
\end{corollary}
\begin{proof}
By \cite[Theorem 3.1]{bashwinger2022non}, $V_d$ (respectively $T_d$) is non-inner amenable and thus $L(V_d)$ (respectively $L(T_d)$)  does not have property Gamma. 

     Then, observe that $ V_{[0,1/d)}$ is weakly malnormal in $V_d$ as $V_{[0,1/d)}\cap V_{[0,1/d)}^{h}=\{1\}$, where $h \in T_d$ is the element 
     \[  h(x) = \begin{cases}  x+ \frac{1}{d}, & 0 \leq x < \frac{d-1}{d} \\
    x-\frac{d-1}{d}, & \frac{d-1}{d} \leq x \leq 1 
    \end{cases}. \]

    Lemma \ref{lemma:VC} implies that
      $L(V_{[0,1/d]})=(L(V_{[0,1/d]})'\cap L(V_{d})^\omega)'\cap L(V_{d})$
     for any ultrafilter $\omega\in \beta \mathbb{N}\setminus \mathbb{N}$, and thus by \cite[Lemma 3.4]{BannonFull2020},  $L(V_{[0, 1/d)})$ is weakly bicentralized in $L(V_d)$.   
     By Proposition \ref{cocycle}, $V_d$ admits an unbounded 1-cocycle into the quasi-regular representation $\ell^2(V_d/ V_{[0,1/d)})$. The assumptions of Theorem \ref{thm:prime} are now satisfied, so $L(V_d)$ is prime.  

      To arrive at the same conclusion for $L(T_d)$, note that $T_{[0,1/d)}$ is weakly malnormal in $T_d$ for the same reason $V_{[0,1/d)}$ is weakly malnormal in $V_d$. Then, by the analogous argument to the previous paragraph, we see that $L(T_d)$ is prime. 
\end{proof}

We end this paper with the prime II$_1$ factors coming from actions of $T_d$ and $V_d.$ The approach is to confirm that both groups are concrete example satisfying \cite[Corollary 1.2]{patchell2023primeness}. We briefly recall how to construct a von Neumann algebra from a generalized Bernoulli action. Let $I$ be a countable set, then given an action $G \curvearrowright I,$ and a tracial von Neumann algebra $B,$ we define $B^I = \otimes_{i \in I}B.$ An element $x \in B^I$ is written $\otimes_{i \in I}x_i,$ where $x_i \in B^i.$ Then $G$ acts on $B^I$ by $g\cdot (\otimes_{i \in I}x_i) = \otimes_{i\in I}x_{g^{-1}i}.$ The action is then extended to all of $B^I$ so that the crossed product von Neumann algebra $M = B^I \rtimes G$ can be formed. 

\begin{corollary}[Theorem \ref{MainThm:Prime2}]\label{cor:Prime2}
    Let $(B, \tau)$ be a II$_1$ factor. The natural action of $V_d \curvearrowright V_d/V_{[0,1/d)}$ gives rise to a prime II$_1$ factor, $M = B^{V_d/V_{[0,1/d)}} \rtimes V_d$ for every integer $d\geq 2$.  Similarly, $N = B^{T_d/T_{[0,1/d)}} \rtimes T_d$ is a prime II$_1$ factor  for every integer $d\geq 2$. 
\end{corollary}

    \begin{proof}
          Note that $T_d$ and $V_d$ are non-inner amenable by \cite{bashwinger2022non} and $V_{[0,1/d)} < V_d$ (respectively $T_{[0,1/d)} < T_d$ ) is weakly malnormal by our discussion above.  In Lemma \ref{lem:notcoamen}, we show that $V_{[0,1/d)}$  (respectively $T_{[0,1/d)}$) is not co-amenable inside $V_d$ (respectively $T_d$).  To see $V_{[0,1/d)}$ has trivial normal core, it is enough to show the action of $V_d$ on $V_d/V_{[0,1/d)}$ is faithful.  To see this, we can use our work in constructing the cocycle in Section \ref{sec:Cocycles} to confirm the action is faithful.  In particular, for each $v \in V_d,$ take its representative triple $T = (\mathcal{P}_1, \mathcal{P}_2, \phi),$ then by choosing an element $gV_{[0,1/d)} \in X_{\mathcal{P}_1},$ the proof of Theorem \ref{cocycle}, discusses how $vgV_{[0,1/d)}$ is in $X_{\mathcal{P}_2}.$  Therefore, the action is faithful and all of the conditions of \cite[Corollary 1.2]{patchell2023primeness} are satisfied. 
        Note that the analogous argument shows $T_{[0,1/d)}$ has trivial normal core, and thus gives the conclusion of the corollary for $T_d$ as well.
    \end{proof}

\bibliographystyle{amsalpha}
\bibliography{stable}

\end{document}